\documentclass[11pt]{article}

\usepackage{mathptmx, latexsym, amsmath, amssymb, color}
\input{epsf.sty}

\usepackage{color}

\newtheorem{theorem}{Theorem}[section]
\newtheorem{remark}{Remark}[section]

\newtheorem{proposition}{Proposition}[section]
\newtheorem{corollary}{Corollary}[section]
\newenvironment{proof}{\begin{trivlist}
\item[\hspace{\labelsep}{\bf\noindent Proof. }]}
{$\hfill\Box$\end{trivlist}}
\title{\bf A multispecies birth--death--immigration process and its diffusion approximation}

\author{
{\sc Antonio Di Crescenzo}\footnote{
Dipartimento di Matematica, Universit\`a di Salerno, Via Giovanni Paolo II n.\ 132, 84084 Fisciano (SA), Italy, 
email: {adicrescenzo@unisa.it}           
}
, \ 
{\sc Barbara Martinucci}\footnote{
Dipartimento di Matematica, Universit\`a di Salerno, Via Giovanni Paolo II n.\ 132, 84084 Fisciano (SA), Italy, 
email: {bmartinucci@unisa.it}  
}
, \ 
{\sc Abdelaziz Rhandi}\footnote{ 
Dipartimento di Ingegneria dell'Informazione, Ingegneria Elettrica e Matematica Applicata, 
Universit\`a di Salerno, Via Giovanni Paolo II n.\ 132, 84084 Fisciano (SA), Italy, 
email: {arhandi@unisa.it}
}
}
\date{\normalsize 
\bf First published in {\em Journal of Mathematical Analysis and Applications}  \\
Vol.\ 442, p.\ 291--316  \ \copyright\ 2016 by Elsevier}

\begin{document}
\maketitle

\begin{abstract}
We consider an extended birth-death-immigration process defined on a lattice formed by the integers 
of $d$ semiaxes joined at the origin. When the process reaches the origin, then it may jumps toward 
any semiaxis with the same rate. The dynamics on each ray evolves according to a one-dimensional 
linear birth-death process with immigration. We investigate the transient and asymptotic behavior of 
the process via its probability generating function. The stationary distribution, when existing, is a 
zero-modified negative binomial distribution. We also study a diffusive approximation of the process, 
which involves a diffusion process with linear drift and infinitesimal variance on each ray. 
It possesses a gamma-type transient density admitting a stationary limit. 
\par
As a byproduct of our study, we obtain a closed form of the number of permutations with a 
fixed number of components, and  a new series form of the polylogarithm function 
expressed in terms of the Gauss hypergeometric function.


\end{abstract} 

\section{Introduction}
We study a continuous-time stochastic process describing the dynamics of a population formed by 
a fixed number $d$ of non-interacting species competing for a single habitat. 
The problem of species competition is often approached in the literature by use of spatial models 
and competitive hierarchy. In some cases the number of species and the number of sites are 
fixed by assumption. See  Buttel {\em et al.} \cite{BuDuLe2002}, where specific attention is 
given to the number of species that can coexist on a finite number of sites. The models 
studied in Di Crescenzo {\em et al.} \cite{DiGiNoRi01} take into account colonization, death 
and replacement, both in the presence and in absence of hierarchic rules for the species. 
In this paper we investigate a continuous-time stationary Markov chain  ${\cal N}$, over a 
lattice formed by the integers of $d$ semiaxes joined at the origin, i.e.\ an extended star 
graph. This process is a suitable extension of a linear birth-death-immigration process 
(with constant immigration rates, and linear birth and death rates), and 
describes the dynamics of a population formed by $d$ non-interacting species into a given 
habitat. As soon as the habitat is occupied by an individual of a certain species (by effect 
of immigration) then the dynamics evolves according to a linear birth-death-immigration 
process until extinction. Next the habitat can be occupied again due to immigration of an 
individual of a possibly different species, and so on. In other terms, the local population 
is sustained primarily by reproduction of resident individuals, but may be subsidized by 
immigration of individuals of the same species. The rule of immigration, however, is seen 
mainly as allowing recolonization of the habitat after the extinction of the local population. 
\par
The linear birth-death process with immigration 
is often employed as a stochastic model for population processes in biology and ecology 
(see, for instance, Chao and Zheng \cite{ChZh2003}, Crawford and Suchard \cite{CrSu2012}, 
Kyriakidis \cite{Ky94}, Ricciardi \cite{Ri86}, Zheng {\em et al.}\ \cite{ZhChJi2004}). 
A  birth-death-immigration process including the possibility of multiple immigrations has been 
discussed recently by Jakeman and Hopcraft \cite{JaHo2012}. We recall
for instance the application of birth-death processes on graphs to evolutionary models 
of spatially structured populations. See Allen and Tarnita \cite{AlTe2014} for a comprehensive 
investigation on state-dependent birth-death population models with fixed population size 
and structure, and Broom and Rycht\'a$\check{\rm r}$ \cite{BrRy2008} for evolutionary
dynamics of populations on graphs. 
%
\par
In this paper we first propose to investigate the distribution of the number of individuals of the local 
population, with special attention to the dependence of the stationary distribution on the number 
of species. We point out that the linear nature of birth and death rates allows us to obtain 
explicit closed-forms both for the transient and stationary dynamics of the process ${\cal N}$, 
rather than approximate or simulated results. 
\par
A further object of our investigation is the diffusive approximation of ${\cal N}$. The 
adopted procedure leads to a diffusion process with linear drift and infinitesimal variance, defined 
on the rays of a star graph. It is worth pointing out that we are able to study the transient behavior 
of the approximating diffusion process, via a gamma density with constant shape parameter 
and time-varying rate.
\par
We recall that diffusion processes on graphs have been studied by several authors. 
See for instance Freidlin and Wentzell \cite{FrWe93}, that is one of the first contributions 
on this topic, and Weber \cite{We01} for occupation time functionals for diffusion 
processes and birth-death processes on graphs. 
An investigation involving a diffusion process on star graph has been performed in
Papanicolaou {\em et al.} \cite{PaPaLe12}, where the authors obtain exit probabilities and
certain other quantities involving exit and occupation times for a Brownian Motion on star 
graph. Other examples of diffusion processes on star graphs can be found in
Mugnolo {\em et al.} \cite{MuPrRh14}. 
\par
This is the plan of the paper. In Section \ref{Section:2} we introduce the process ${\cal N}$ 
and its generator.
In Section \ref{Section:GenFunct} we develop a generating function-based approach and obtain 
some useful integral equations. This allows us to get a formal expression for the transient probability 
that the process is located in the origin, i.e.\ the habitat is empty, the proof being provided 
in \ref{appendix-A}. In Section \ref{section:3} we perform the transient analysis of the 
process in two different cases. The adopted technique is based on the coupling of the homotopy 
perturbation method and an expansion in Taylor series. In Section \ref{section:4} we use the 
Laplace transform to derive the asymptotic expression of the state probabilities (involving a 
zero-modified negative binomial distribution) and also of the mean and variance. 
Section \ref{section:5} deals with the diffusive approximation of ${\cal N}$. We adopt a 
customary scaling that leads to a time-homogeneous diffusion process on the star graph, 
characterized by linear infinitesimal moments. A gamma-type stationary density is also 
obtained under suitable assumptions. In Section \ref{section:new} we interpret our 
results in biological terms with special reference to the role of the number of species $d$. 
Finally, some concluding remarks are given in Section \ref{section:8}. 
\par
It is worth pointing out that, as a byproduct of our investigations, in Section \ref{section:3} 
we provide some new results of wide interest in mathematics, i.e.\ a closed form of the 
number of permutations of  $\{1,\ldots , n\}$ with $k$ components, also known as the number 
of permutations  with $k-1$ global descents, and  a new series form of the polylogarithm 
function expressed in terms of the Gauss hypergeometric function. 
\section{The stochastic model}\label{Section:2}
Consider a habitat that may accommodate individuals of 1 out of $d$ population species, 
with $d\in \mathbb{N}^+$, and let $D=\{1,2,\ldots,d\}$. 
Assume that the evolution of individuals in the habitat is subject to 
births, deaths and immigrations, according to the following rules, where $h>0$ is sufficiently small: 
\begin{description}
\item{\em (i) \ } 
If the habitat is empty at time $t$, then during the time interval $(t,t+h]$ either the 
habitat is occupied by an individual of species $j$, $(j\in D)$, with 
probability $\alpha h+o(h)$ (due to immigration), or it remains empty with 
probability $1-d\alpha h+o(h)$. 
\item{\em (ii) \ } 
If the habitat at time $t$ is occupied by $k$ individuals of species $j$, $(j\in D)$, then 
during the time interval $(t,t+h]$ either one individual dies with probability $\mu k  h+o(h)$, 
or a new individual of the same species arrives with probability $(\alpha+\lambda k)  h+o(h)$  
(by the effect of immigration or birth), or the population size remains unchanged with probability
$1-(\alpha+(\lambda+\mu) k) h+o(h)$. 
\end{description}
Hence, note that when the habitat is empty each species may compete for the colonization of the 
habitat, whereas when a species occupies the habitat there is no interaction with other species. 
\begin{figure}[h]
\begin{center}
\begin{picture}(260,220)
\put(50,50){\line(1,1){140}}
\put(50,50){\line(2,1){140}}
\put(50,50){\line(6,1){140}}
\put(50,50){\line(2,-1){140}}
\put(50,50){\circle*{3}}
\put(100,100){\circle*{3}}
\put(100,75){\circle*{3}}
\put(100,58.5){\circle*{3}}
\put(100,25){\circle*{3}}
\put(150,150){\circle*{3}}
\put(150,100){\circle*{3}}
\put(150,66.5){\circle*{3}}
\put(150,0){\circle*{3}}
\put(40,53){\makebox(20,15)[t]{\small $(0,0)$}}
\put(91,102){\makebox(20,15)[t]{\small $(1,1)$}}
\put(91,75){\makebox(20,15)[t]{\small $(1,2)$}}
\put(91,55){\makebox(20,15)[t]{\small $(1,3)$}}
\put(91,22){\makebox(20,15)[t]{\small $(1,d)$}}
\put(141,152){\makebox(20,15)[t]{\small $(2,1)$}}
\put(141,99){\makebox(20,15)[t]{\small $(2,2)$}}
\put(141,65){\makebox(20,15)[t]{\small $(2,3)$}}
\put(141,-1){\makebox(20,15)[t]{\small $(2,d)$}}
\put(205,180){\makebox(20,15)[t]{\small $S_1$}}
\put(205,110){\makebox(20,15)[t]{\small $S_2$}}
\put(205,65){\makebox(20,15)[t]{\small $S_3$}}
\put(205,-30){\makebox(20,15)[t]{\small $S_d$}}
\put(100,40){\line(0,1){2}}
\put(100,45){\line(0,1){2}}
\put(100,50){\line(0,1){2}}
\put(150,26){\line(0,1){2}}
\put(150,38){\line(0,1){2}}
\put(150,50){\line(0,1){2}}
\put(215,5){\line(0,1){4}}
\put(215,25){\line(0,1){4}}
\put(215,45){\line(0,1){4}}
\end{picture}
\vspace{1cm}
\end{center}
\caption{Schematic representation of the state space $S$.}
\label{FigSystem}
\end{figure}
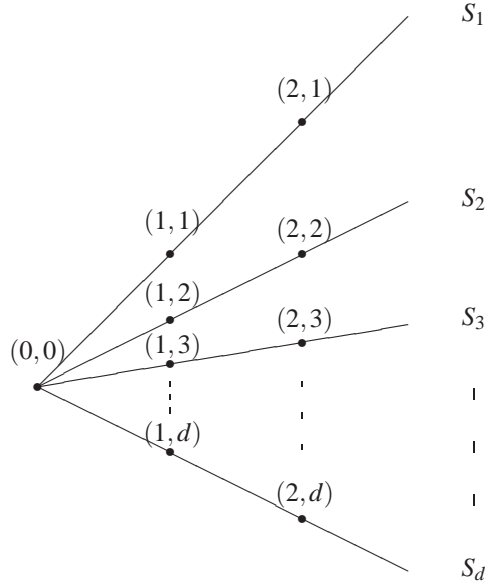
\par
The dynamics  is described by a continuous-time Markov chain 
${\cal N}:=\{(N(t),J(t)), t\geq 0\}$, where $(N(t),J(t))=(0,0)$ if at time $t$ the habitat is empty, 
and $(N(t),J(t))=(k,j)$ if at time $t$ the habitat is occupied by $k$ individuals of the $j$-th 
species. The state space of ${\cal N}$ is the set $S=\{(0,0)\}\cup ({\mathbb N}^+\times D)$, 
consisting of the integers of $d$ semiaxes $S_1,S_2,\ldots, S_d$ $(d\in\mathbb{N}^+)$ 
with a common origin $(0,0)$ (see Figure \ref{FigSystem}). 
We denote the transition rates of ${\cal N}$ by 
$$
 q({\bf u};{\bf v})=\lim_{h\rightarrow 0^+} \frac{1}{h} 
 {\mathbb P}[(N(t+h),J(t+h))={\bf v}\,|\,(N(t),J(t))={\bf u}],
 \qquad {\bf u},{\bf v}\in S.
$$
According to assumptions {\em (i)} and {\em (ii)}, the generator 
$Q:=(q({\bf u};{\bf v}),{\bf u},{\bf v}\in S)$ of ${\cal N}$ satisfies 
\begin{equation}
 q(0,0;1,j)=\alpha, \qquad q(k,j;k+1,j)=\alpha+\lambda k, 
 \qquad q(k,j;k-1,j)=\mu k,
  \nonumber 
\end{equation}
\begin{equation}
 q(k,j;r,j)=0\quad \hbox{if $|k-r|>1$}, \qquad q(k,i;r,j)=0\quad \hbox{if $i\neq j$}, 
 \label{eq:tassi}
\end{equation}
$$ 
 q(0,0;r,j)=q(r,j;0,0)=0\quad \hbox{if $r\neq 1$}, 
$$
for all $k,r\in {\mathbb N}^+$ and $i,j\in D$, where $\alpha>0$, $\lambda>0$ and $\mu>0$ 
are constants denoting the immigration, birth and death rate per individual, respectively. 
\par
Noting that ${\cal N}$ is a skip-free process and that $(0,0)$ is a non-absorbing state, the above 
assumptions imply that ${\cal N}$ is nonexplosive (cf.\ Chen {\em et al.} \cite{ChPoZhCa2005}), 
and hence uniquely determined by $Q$. 
\par
If $d=1$ then ${\cal N}$ identifies with the linear birth-death process with immigration, 
that is well-known among population models (see Section 3.1 of Crawford and Suchard \cite{CrSu2012}
and references therein). The purpose of our study is the extension of the birth-death-immigration 
process to the case of $d$ non-iteracting populations according to the assumptions indicated above. 
\section{Generating functions}\label{Section:GenFunct}
We assume that the initial state of ${\cal N}$ is the origin, which for simplicity will be henceforth 
denoted as $0$ instead of $(0,0)$ . Hence, the transition probability of ${\cal N}$ is defined as 
\begin{equation}
\begin{split}
  & p(k,j, \cdot )={\mathbb P}\{(N(\cdot),J(\cdot))=(k, j)\,|\, (N(0),J(0))=0\},\qquad 
  k\in {\mathbb N}^+, \;\; j\in D, \\
  & p(0, \cdot )={\mathbb P}\{(N(\cdot),J(\cdot))=0\,|\, (N(0),J(0))=0\}.
\end{split}
\label{eq:probpkjt}
\end{equation}
The initial condition is thus expressed as 
\begin{equation}
 \lim_{t\to 0^+}p(0,t)=1.
 \label{eq:probiniz}
\end{equation}
From (\ref{eq:probpkjt}) we have that 
\begin{equation}
 P(k,\cdot):=\sum_{j=1}^d p(k,j,\cdot),\qquad k\in {\mathbb N}^+, 
 \label{pgrandekt}
\end{equation}
is the probability of occupancy of the $k$-th state of any semiaxis. 
\par
Consider the probability generating function 
\begin{equation}
F(z,t)=p(0,t)+\sum_{k\geq 1} z^k P(k,t),\qquad z\in [0,1],\quad t\geq 0.
\label{FPgrande}
\end{equation}
By virtue of (\ref{eq:probiniz}), it satisfies the initial condition
\begin{equation}
F(z,0)=1,\qquad z\in [0,1].
\label{initialconditions}
\end{equation}
Moreover, the following boundary conditions hold:
\begin{equation}
F(1,t)=1,\qquad t\geq 0,
\label{boundconditions}
\end{equation}
\begin{equation}
F(0,t)=p(0,t),\qquad t\geq 0,
\label{initialconditions2}
\end{equation}
where $p(0,t)$ is the probability that the habitat is empty at time $t$. 
\begin{proposition}
The generating function (\ref{FPgrande}) satisfies the following differential equation
for $z\in [0,1]$ and $t\geq 0$:
\begin{equation}
{\partial \over \partial t}\!F(z, t)=-\alpha (d-1)(1-z)p(0,t)-\alpha (1-z) F(z,t)
-(\lambda z-\mu)(1-z) {\partial \over \partial z}\!F(z, t).
\label{eq:diffF}
\end{equation}
\end{proposition}
\begin{proof}
Due to (\ref{eq:tassi}), for $k\in {\mathbb N}^+$ and $j\in D$, the following
system of differential-difference equations holds for $t>0$:
\begin{eqnarray}
\label{eq:system}
&&  \hspace{-0.8cm}
{d \over d t}\;p(0, t)=\mu \sum_{j=1}^d p(1,j,t)-d \alpha\,p(0,t),
\\
&&  \hspace{-0.8cm}
{d \over d t}\;p(k,j, t)=[\alpha+\lambda(k-1)]\,p(k-1,j, t)
+\mu (k+1)\,p(k+1,j, t) -[\alpha+(\lambda+\mu)k]\,p(k,j, t).
\nonumber
\end{eqnarray}
Hence, the probability generating function
\begin{equation}
 G_j(z,t):=\sum_{k\geq 1} z^k p(k,j, t),\qquad z\in [0,1],\quad t\geq 0,
 \label{eq:defGj}
\end{equation}
for $j\in D$ satisfies the following differential equation:
$$
{\partial \over \partial t} G_j(z, t)
=\mu  p(1,j,t)+ \alpha z p(0,t)-\alpha (1-z) G_j(z,t)
-(\lambda z-\mu)(1-z){\partial \over \partial z}G_j(z, t).
$$
Hence, the proof of (\ref{eq:diffF}) follows recalling Eqs.\ (\ref{FPgrande}) and (\ref{eq:defGj}). 
\end{proof}
\par
Here, and in the following, $f'$ denotes the derivative of any function $f$.
\begin{proposition}
Eq.\ (\ref{eq:diffF}), with conditions (\ref{boundconditions}) and (\ref{initialconditions}),
admits the following solution for $z\in [0,1]$ and $t\geq 0$:
\begin{equation}
F(z,t)=H(t)+(d-1)\int_{0}^{t} H'(t-y) p(0,y){\rm d}y,
\label{soluzioneFgenerale}
\end{equation}
where
\begin{equation}
 H(t)=h(t,z;\lambda,\mu):=\left\{
 \begin{array}{ll}
 \displaystyle{\frac{(\lambda-\mu)^{\frac{\alpha}{\lambda}} {\rm e}^{-\frac{\alpha }{\lambda}(\lambda-\mu)t}}
{\left[\lambda(1-z)-(\mu-\lambda z)\,{\rm e}^{-(\lambda-\mu) t}\right]^{\frac{\alpha}{\lambda}}}}, & \lambda\neq \mu,
\\[0.5cm]
\displaystyle{\frac{1}{\left[1+\lambda t (1-z)\right]^{\frac{\alpha}{\lambda}}}},
& \lambda=\mu.
 \end{array}
 \right.
\label{funzioneh}
\end{equation}
\label{propFzt}
\end{proposition}
\begin{proof}
Let us adopt the method of characteristics.
If $\lambda\neq \mu$, Eq. (\ref{eq:diffF}) can be rewritten as 
\begin{equation}
(\lambda z-\mu) (1-z) {\partial F \over \partial z}+{\partial F \over \partial t}
+\alpha (1-z) F+\alpha(d-1)(1-z)p(0,t)=0,
\label{PDEFLM}
\end{equation}
which gives the following characteristic equations for the original system
\begin{equation}
 {\displaystyle{d z \over d s}=(\lambda z-\mu) (1-z)}, \qquad 
 {\displaystyle{d t \over d s}=1}, \qquad 
 {\displaystyle{d F \over d s}=-\alpha(1-z)F-\alpha(d-1)(1-z)p(0,t)}.
\label{eq:characteristicLM}
\end{equation}
From Eq.\ (\ref{eq:characteristicLM}), along the characteristic curves
$$
z=1-\frac{(\lambda-\mu)(1-\tau)}{\lambda (1-\tau)-(\mu-\lambda \tau){\rm e}^{(\lambda-\mu)s}},\qquad t=s,\qquad \tau\in {\mathbb R},
$$
the partial differential equation (\ref{PDEFLM}) and conditions (\ref{initialconditions}) and (\ref{boundconditions})
yield
$$
 \displaystyle \frac{{\rm d} F}{{\rm d}s}+\alpha
 \left[\frac{(\lambda-\mu)(1-\tau)}{\lambda (1-\tau)-(\mu-\lambda \tau){\rm e}^{(\lambda-\mu)s}}\right]F
+\alpha (d-1)\displaystyle
\left[\frac{(\lambda-\mu)(1-\tau)}{\lambda (1-\tau)-(\mu-\lambda \tau){\rm e}^{(\lambda-\mu)s}}\right]
 p(0,s)=0,
$$
with $F(0)=1$. 
Hence, Eqs.\ (\ref{soluzioneFgenerale}) and (\ref{funzioneh}) follow after some calculations.
If $\lambda=\mu$, the proof is similar.
\end{proof}
\par
Hereafter we show that $p(0,t)$
satisfies a linear Volterra integral equation of the 2nd kind.
\begin{corollary}
The following renewal equation holds, for $t> 0$,
\begin{equation}
 p(0,t)=1-G(t)-(d-1)\int_{0}^{t} G'(t-y) p(0,y){\rm d}y,
 \label{eq:eqpot}
\end{equation}
where
\begin{equation}
 G(t)=1-h(t,0;\lambda,\mu),
 \label{eq_def1menoGt}
\end{equation}
with $h(t,z;\lambda,\mu)$ defined in (\ref{funzioneh}).
\end{corollary}
\begin{proof}
Eqs.\ (\ref{eq:eqpot}) and (\ref{eq_def1menoGt}) follow from Eqs.\ (\ref{soluzioneFgenerale}) and
(\ref{funzioneh}), for $z=0$ and recalling  condition (\ref{initialconditions2}).
\end{proof}
\par
We remark that $G(t)$, given in (\ref{eq_def1menoGt}), is a proper
distribution function when $\lambda\geq \mu$.
\par
Hereafter we consider the distribution function
\begin{equation}
F_Y^{(j)}(t):={\mathbb P}(Y_1+Y_2+\cdots+Y_j\leq t),
\label{distrsomma}
\end{equation}
where $Y_1,Y_2,\ldots,Y_j$ is a sequence of i.i.d.\ random variables. In the following
theorem we give a formal representation of $p(0,t)$ in terms of (\ref{distrsomma})
when $Y_i$'s are nonnegative and have a specific distribution.
\begin{theorem}
For $t\geq 0$ we have
\begin{equation}
p(0,t)=\left\{
\begin{array}{ll}
\displaystyle{1-{d}\sum_{j=1}^{+\infty} (1-d)^{j-1} F_Y^{(j)}(t)},
 & \quad \lambda\geq \mu, \\[.3cm]
\displaystyle{1-{d}\sum_{j=1}^{+\infty} (1-d)^{j-1}
\left[1-\left(\frac{\mu-\lambda}{\mu}\right)^{\frac{\alpha}{\lambda}}\right]^j F_Y^{(j)}(t)},
 & \quad \lambda< \mu, \\[.4cm]
\end{array}
\right.
\label{p0}
\end{equation}
where
\begin{equation}
F_Y^{(1)}(t)=\left\{
\begin{array}{ll}
\displaystyle{1-\left(\frac{\lambda-\mu}{\lambda {\rm e}^{(\lambda-\mu) t}-\mu}\right)^{\frac{\alpha}{\lambda}}}, & \quad \lambda> \mu, \\[.3cm]
\displaystyle{1-\frac{1}{(1+\lambda t)^{\frac{\alpha}{\lambda}}}},  &
\quad \lambda= \mu, \\[.4cm]
\displaystyle{\frac{1}{1-\left(1-\frac{\lambda}{\mu}\right)^{\frac{\alpha}{\lambda}}}
\, \left[1-\left(\frac{\mu-\lambda}{\mu-\lambda {\rm e}^{-(\mu-\lambda) t}}\right)^{\frac{\alpha}{\lambda}}\right]  },  & \quad \lambda< \mu. \\[.4cm]
\end{array}
\right.
\label{distr_somma}
\end{equation}
\label{theorem1}
\end{theorem}
\par
The proof of Theorem \ref{theorem1} is given in \ref{appendix-A}.
\begin{remark}
{\rm
The right-hand-side of Eq.\ (\ref{distr_somma}), in each of the three cases, identifies
with the distribution function of suitable transformations of a random variable, say $Z$,
having Pareto type II (Lomax) distribution with shape and scale  parameters $\tilde{\alpha}$
and $\tilde{\lambda}$, respectively.  Namely, \\
$\bullet$ \ if $\lambda>\mu$, then $F_Y^{(1)}(t)$ is the distribution function of  $\log{(Z+1)}/(\lambda-\mu)$,
with $\tilde{\alpha}=\alpha/\lambda$ and $\tilde{\lambda}=(\lambda-\mu)/\lambda$, \\
$\bullet$ \ if $\lambda=\mu$,  then $F_Y^{(1)}(t)$ is the distribution function of $Z$ for
$\tilde{\alpha}=\alpha/\lambda$ and $\tilde{\lambda}=1/\lambda$, \\
$\bullet$ \ if $\lambda<\mu$,  then $F_Y^{(1)}(t)$ is the distribution function of $-\log{(1-Z)}/(\mu-\lambda)$,
assuming that $Z$ has support $(0,1)$ and parameters $\tilde{\alpha}=\alpha/\lambda$
and $\tilde{\lambda}=(\mu-\lambda)/\lambda$.
}
\end{remark}
%
%
\section{Transient analysis}\label{section:3}
In this section, for $t$ ranging over specified intervals of $\mathbb R$, we obtain explicit expressions for 
the generating function $F(z,t)$ and for the transient probability $p(0,t)$ that the habitat is empty. 
We also study the cumulative probability $P(k,t)$, i.e.\ the probability that the habitat is occupied 
by $k$ individuals (irrespective of their species) at time $t$. We consider 2 cases:  \\
1.  Immigration, birth and death rates are equal ($\alpha=\lambda=\mu$). \\
2.  Immigration and birth rates are equal, the death rate is different  ($\alpha=\lambda$ 
and $\mu\neq \lambda$).
\subsection{Transient analysis for $\alpha=\lambda=\mu$}
Aiming to obtain an expression for $p(0,t)$ when  $\alpha=\lambda=\mu$, 
let us denote by $t_{n,k}$ the number of permutations of $\{1,\ldots,n\}$, $n\geq 1$, 
with $k\geq 1$ components (see, for instance, Comtet \cite{Comtet}, p.\ $262$ and \cite{OEIS}). 
Alternatively, $t_{n,k}$ is the
number of permutations of $\{1,\ldots,n\}$ with $k-1$ global descents. Permutations with one component,
i.e.\ $t_{n,1}$, are known as  indecomposable permutations (we recall that a permutation is called
indecomposable if its one-line notation cannot be split into two parts such that every number in the first
part is smaller than every number in the second part). Noting that $t_{n,k}=0$ if $n<k$, an implicit 
recursion formula for $t_{n,k}$ is given by (see Propositions 2.4 and 2.7 of \cite{HegMart2014}) 
\begin{equation}
t_{n,k}=\left\{
 \begin{array}{ll}
 n!-\displaystyle{\sum_{j=1}^{n-1} (n-j)!\cdot t_{j,1}}, & k=1, \\[0.6cm]
\displaystyle{\sum_{j=1}^{n-k+1} t_{j,1}\cdot t_{n-j,k-1}}, & 2\leq k\leq n.
 \end{array}
 \right.
\label{numberstnk}
\end{equation}
\begin{proposition}\label{prop1}
Let $\alpha=\lambda=\mu$. If $0<t<1/\lambda$ the integral equation (\ref{eq:eqpot}) admits
the following solution:
\begin{equation}
 p(0,t)=1+\sum_{n=1}^{+\infty}\frac{(-\lambda t)^n}{n!}\sum_{j=1}^n t_{n,j}\, d^j.
 \label{solp0}
\end{equation}
\end{proposition}
\begin{proof}
The proof is based on the coupling of the homotopy perturbation method and the expansion
of the involved functions as Taylor series (see Biazar and Eslami \cite{BiazarEslami11}).
From (\ref{eq:eqpot}), for $\alpha=\lambda=\mu$, we can construct the following homotopy
\begin{equation}
H(p,q)=p(0,t)-\frac{1}{(1+\lambda t)}+\lambda (d-1) q \int_{0}^t
\frac{p(0,y)}{[1+\lambda (t-y)]^2}{\rm d}y=0,
\label{eq:homotopy1}
\end{equation}
with the embedding parameter $q$. By assuming that $p(0,t)=\sum_{n=0}^{+\infty} q_n(t) q^n$
and substituting functions $\frac{1}{(1+\lambda t)}$ and $\frac{1}{[1+\lambda (t-y)]^2}$ by their
Taylor series forms, in agreement with Eq.\ (\ref{eq:homotopy1}) we define
\begin{equation}
\widetilde H(p,q)=\sum_{n=0}^{+\infty} q_n(t) q^n-\sum_{n=0}^{+\infty} (-\lambda t)^n q^n+\lambda (d-1)
\sum_{n=0}^{+\infty} q^{n+1} \int_{0}^t \alpha_n(y,t) {\rm d}y=0,
\label{eq:homotopy2}
\end{equation}
for $0<t<1/\lambda$, with
\begin{equation}
\alpha_n(y,t)=\sum_{k=0}^n q_k(y)(n-k+1) [-\lambda(t-y)]^{n-k}.
\label{eq:alfa}
\end{equation}
Hence, equating the coefficients of the terms with identical powers of $q$, we find that
function $q_n(x)$ is solution of the following recursive equation:
\begin{equation}
q_n(x)=(-\lambda x)^n-\lambda (d-1) \int_{0}^{x} \alpha_{n-1}(y,x) {\rm d}y,
\qquad n\in\mathbb{N}^+,
\label{eq:qnrecursive}
\end{equation}
with $q_0(x)=1$. By direct calculations, from Eq.\ (\ref{eq:qnrecursive}) one immediately gets
$$
 q_1(x)=-d \lambda  x.
$$
Hereafter we make use of the strong induction principle to show that
\begin{equation}
q_n(x)=\frac{(-\lambda x)^n}{n!}\sum_{j=1}^n t_{n,j}\, d^j,
\qquad n\in\mathbb{N}^+.
\label{eq:qn}
\end{equation}
Being $t_{k,k}=1$ for all $k\geq 1$ (see \cite{Comtet}, p.\ $262$), Eq.\ (\ref{eq:qn}) holds for
$n=1$. Assuming that  (\ref{eq:qn}) holds for all $k=1,2,\ldots, n$ we now
prove that it holds true for $k=n+1$.
From Eq.\ (\ref{eq:alfa}), due to the induction hypothesis, we have
$$
\int_{0}^{x} \alpha_{n}(y,x) {\rm d}y=x \frac{(-\lambda x)^{n}}{(n+1)!} \sum_{j=1}^{n} (n+1-j)! \sum_{r=1}^{j} t_{j,r}\, d^r
+x (-\lambda x)^{n}, \qquad n\in\mathbb{N}^+.
$$
Hence, recalling Eq.\ (\ref{eq:qnrecursive}) and using  $t_{k,k}=1$ $\forall k\geq 1$, we obtain
\begin{eqnarray}
&& \hspace*{0.cm}
q_{n+1}(x)=\frac{(-\lambda x)^{n+1}}{({n+1})!} \left\{d \,({n+1})!
+(d-1) \sum_{j=1}^{n} (n+1-j)! \sum_{r=1}^{j} t_{j,r}\, d^r  \right\}
\label{eq:qn1}
\\
&& \hspace*{1.2cm}
=\frac{(-\lambda x)^{n+1}}{(n+1)!} \left\{ d \left[(n+1)!-\sum_{j=1}^{n} (n+1-j)!\,t_{j,1}\right]+d^{n+1}
\right.
\nonumber
\\
&& \hspace*{1.3cm}
\left.
+\sum_{s=2}^{n} d^s \left[(n-s+2)! +\sum_{j=s}^{n} (n+1-j)! (t_{j,s-1}-t_{j,s})\right]
\right\}.
\nonumber
\end{eqnarray}
Recalling Eq.\ (\ref{numberstnk}), we note that
$$
\sum_{j=s}^{n} (n+1-j)! \cdot t_{j,s}
=\sum_{r=s-1}^{n-1} t_{r,s-1} \sum_{j=r+1}^{n} (n+1-j)! \cdot t_{j-r,1}.
$$
Hence, repeated applications of Eq.\ (\ref{numberstnk}) yield
\begin{eqnarray}
&& \hspace*{-0.9cm}
(n-s+2)! t_{s-1,s-1}+\sum_{j=s}^{n} (n+1-j)! (t_{j,s-1}-t_{j,s})
\nonumber
\\
&& \hspace*{-0.6cm}
=(n-s+2)!+\sum_{r=s}^{n-1} t_{r,s-1} \Bigg[(n+1-r)!-\sum_{j=r+1}^{n} (n+1-j)! t_{j-r,1}\Bigg]
+\,t_{n,s-1}-\sum_{j=s}^{n} (n+1-j)! t_{j-s+1,1}
\nonumber
\\
&& \hspace*{-0.6cm}
=(n-s+2)!+\sum_{r=s}^{n} t_{r,s-1}\cdot t_{n+1-r,1}-\sum_{j=s}^{n} (n+1-j)! t_{j-s+1,1}
\nonumber
\\
&& \hspace*{-0.6cm}
=t_{n-s+2,1}+\sum_{r=s}^{n} t_{r,s-1}\cdot t_{n+1-r,1}=t_{n+1,s}.
\label{eq:qn2}
\end{eqnarray}
From Eqs.\ (\ref{eq:qn1}) and (\ref{eq:qn2}) we thus obtain Eq.\ (\ref{eq:qn}).
Finally, by taking $q=1$ in assumption $p(0,t)=\sum_{n=0}^{+\infty} q_n(t) q^n$ we get Eq.\ (\ref{solp0}).
\end{proof}
\begin{proposition}\label{prop2}
If $\alpha=\lambda=\mu$, then for $0<t<1/\lambda$ we have
\begin{eqnarray}
&&  \hspace*{-0.8cm}
F(z,t)=\frac{1}{1+\lambda t(1-z)}-\frac{\lambda t (d-1)(1-z)}{1+\lambda t (1-z)}
\left[1+\sum_{n=1}^{+\infty}\frac{(-\lambda t)^n}{n!}\sum_{j=1}^n t_{n,j}\, d^j\right]
-\frac{(d-1)(1-z)}{[1+\lambda t (1-z)]^2}
\nonumber
\\
&& \hspace*{0.4cm}
\times \sum_{n=1}^{+\infty}n \frac{(-\lambda t)^{n+1}}{(n+1)!} 
\,{}_2F_1\left(1,n+1;n+2;1-\frac{1}{1+\lambda t (1-z)}\right) \sum_{j=1}^n t_{n,j}\, d^j,
\label{funzioneF}
\end{eqnarray}
where
\begin{equation}
 {}_{2}F_{1}(a,b;c;z)=\sum_{n=0}^{+\infty} \frac{(a)_n(b)_n}{(c)_n}\,\frac{z^n}{n!}
\label{Gaussipergeomfunction}
\end{equation}
is the Gauss hypergeometric function. (Here, and in the remainder of the paper,
$(d)_n=d(d+1)(d+2)\cdots(d+n-1)$, $n\geq 1$, denotes the Pochhammer symbol, with $(d)_0=1$ for $d\neq 0$.)
\end{proposition}
\begin{proof}
From Eqs.\ (\ref{soluzioneFgenerale}), (\ref{funzioneh}) and (\ref{solp0}), we obtain
\begin{eqnarray*}
&&   \hspace*{-0.6cm}
F(z,t)=\frac{1}{1+\lambda t(1-z)}-\frac{\lambda t (d-1)(1-z)}{1+\lambda t (1-z)}
-(d-1) \sum_{n=1}^{+\infty}\frac{(-\lambda t)^n}{n!}\sum_{j=1}^n t_{n,j}\, d^j
\nonumber
\\
&& \hspace*{+0.4cm}
+\frac{(d-1)}{1+\lambda t(1-z)} \sum_{n=1}^{+\infty}\frac{(-\lambda t)^n}{n!}
\,{}_2F_1\left(1,n;n+1;1-\frac{1}{1+\lambda t (1-z)}\right)\sum_{j=1}^n t_{n,j}\, d^j.
\end{eqnarray*}
Hence, making use of Eqs.\ $15.2.25$ and $15.1.8$ of \cite{Abram1994},
%
after some calculations we come to Eq.\ (\ref{funzioneF}). 
\end{proof}
\par
We now obtain probability (\ref{pgrandekt}) under the assumptions of
Propositions \ref{prop1} and \ref{prop2}.
\begin{proposition}\label{prop:Pktlugm}
If $\alpha=\lambda=\mu$, for $0<t<1/\lambda$ we have
\begin{eqnarray}
&& \hspace*{-0.5cm}
P(k,t)
=\frac{(\lambda t)^k}{(\lambda t+1)^{k+1}}\left\{d+(d-1) \sum_{n=1}^{+\infty} \frac{(-\lambda t)^{n}}{n!}
\sum_{j=1}^n t_{n,j}\, d^j \right.
\nonumber 
\\
&& \hspace*{0.6cm}
\left.
\times \sum_{r=0}^{+\infty} \frac{(n)_r}{(n+1)_r}\left(1-\frac{1}{1+\lambda t}\right)^r
\,{}_2F_1\left(-r,-k;1;-\frac{1}{\lambda t}\right)\right\}.
\label{probunione}
\end{eqnarray}
\end{proposition}
\begin{proof}
From Eq.\ (\ref{funzioneF}), and recalling (\ref{solp0}) and 
(\ref{Gaussipergeomfunction}), we have
\begin{eqnarray*}
&& \hspace*{-0.9cm}
F(z,t)=\frac{1}{1+\lambda t (1-z)}
-(d-1) p(0,t)\frac{\lambda t (1-z)}{1+\lambda t(1-z)}
\\
&& \hspace*{-0.2cm}
-\frac{(d-1)(1-z)}{[1+\lambda t (1-z)]^2}
\sum_{k=0}^{+\infty} \left[1-\frac{1}{1+\lambda t(1-z)}\right]^k 
\sum_{n=1}^{+\infty} \frac{(-\lambda t)^{n+1}}{(n-1)! (n+k+1)}\sum_{j=1}^n t_{n,j}\, d^j
\\
&& \hspace*{-0.2cm}
=p(0,t)+\sum_{m=0}^{+\infty}z^m \frac{(\lambda t)^m}{(\lambda t+1)^{m+1}}
-p(0,t)\left[\frac{1+d \lambda t}{1+\lambda t}-(d-1) \sum_{m=1}^{+\infty}z^m \frac{(\lambda t)^m}{(\lambda t+1)^{m+1}} \right]\\
&& \hspace*{-0.2cm}
-(d-1) \sum_{m=0}^{+\infty} z^m \frac{(\lambda t)^{m}}{(\lambda t+1)^{m+2}}
\sum_{n=1}^{+\infty} \frac{(-\lambda t)^{n+1}}{(n-1)!} \sum_{j=1}^n t_{n,j}\, d^j
\sum_{k=0}^{+\infty}\frac{\left(\frac{\lambda t}{\lambda t+1}\right)^k }{n+k+1}\,
{}_2F_1\left(-k-1,-m;1;-\frac{1}{\lambda t}\right).
\end{eqnarray*}
Hence, since $p(0,t)$ satisfies the integral equation (\ref{eq:eqpot}), it results
\begin{eqnarray*}
&& \hspace*{-1.cm}
F(z,t)=p(0,t)+\sum_{m=1}^{+\infty}z^m \frac{(\lambda t)^m}{(\lambda t+1)^{m+1}}
+p(0,t) (d-1) \sum_{m=1}^{+\infty}z^m \frac{(\lambda t)^m}{(\lambda t+1)^{m+1}} \\
&& \hspace*{0.2cm}
+(d-1) \sum_{m=1}^{+\infty} z^m \frac{(\lambda t)^{m+1}}{(\lambda t+1)^{m+2}}
\sum_{n=1}^{+\infty} \frac{(-\lambda t)^{n}}{(n-1)!} \sum_{j=1}^n t_{n,j}\, d^j
\\
&& \hspace*{0.2cm}
\times \sum_{k=0}^{+\infty}\frac{(n+1)_k}{(n+2)_k\, (n+1)} \left(\frac{\lambda t}{\lambda t+1}\right)^k \,
{}_2F_1\left(-k-1,-m;1;-\frac{1}{\lambda t}\right).
\end{eqnarray*}
Finally, recalling Eq.\ (\ref{FPgrande}) and equating the coefficients
of the terms with identical powers of $z$ we obtain Eq.\ (\ref{probunione}).
\end{proof}
%
\begin{figure}[h] 

\begin{center}
\epsfxsize=6.5cm
\epsfbox{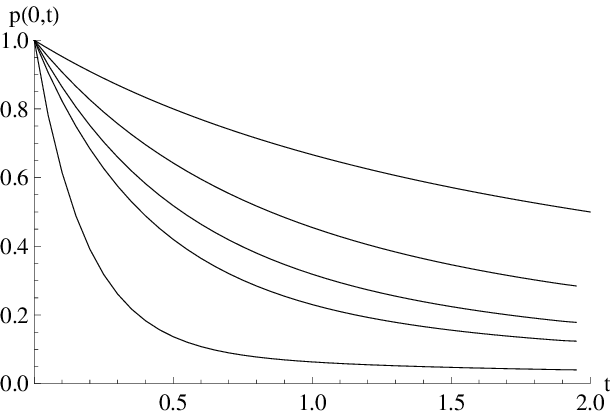}
$\;$
\epsfxsize=6.5cm
\epsfbox{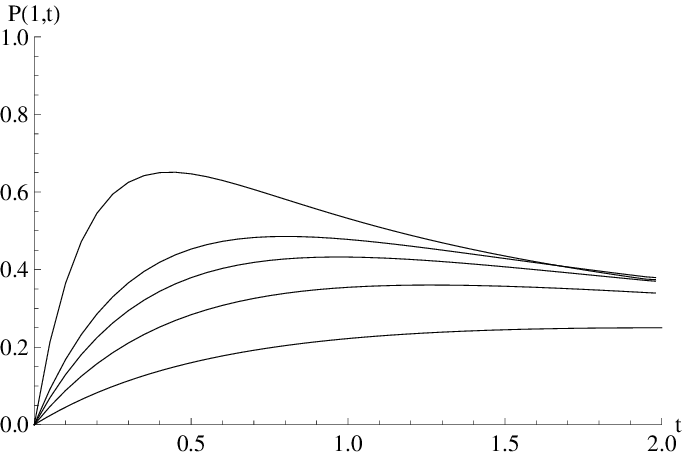}

\caption{Plots of $p(0,t)$ and $P(1,t)$ for $\lambda=0.5$, $\mu=0.5$ and $\alpha=0.5$,
for $d=1,2,3,4,10$, from top to bottom for $p(0,t)$, and from bottom to top for $P(1,t)$.}
\label{fig:LugualeM}
\end{center}
\end{figure}
%
\par
In Figure \ref{fig:LugualeM} we show some plots obtained by means of the expressions given in
Proposition \ref{prop1} and Proposition \ref{prop:Pktlugm}.
\subsection{Transient analysis for $\alpha=\lambda$ and $\mu\neq \lambda$}
In order to investigate the case $\alpha=\lambda$ and  $\mu\neq \lambda$,
we set for brevity 
\begin{equation}
 Q_{j,m}\equiv Q_{j,m}\left(\frac{\mu}{\lambda}\right)
 :=\sum_{\substack{s_1,\ldots,s_j\geq 2\\ s_1+\cdots+s_j=m}}
 A_{s_1}\left(\frac{\mu}{\lambda}\right)\times \cdots \times A_{s_j}\left(\frac{\mu}{\lambda}\right),
\label{eq:defsum}
\end{equation}
where $A_n(t)$ are the Eulerian Polynomials (see, for instance, Foata \cite{Foata} or Hirzebruch \cite{Hi2008}).
\begin{proposition}\label{p0LneqM}
If $\alpha=\lambda$ and $\mu\neq \lambda$, then for $0<t<\log(\mu/\lambda)/(\mu-\lambda)$ the integral
equation (\ref{eq:eqpot}) admits the following solution
\begin{eqnarray}
&& \hspace*{-1cm}
p(0,t)=1-d \left\{1-\frac{\mu-\lambda}{\mu-\lambda {\rm e}^{-(\mu-\lambda)t}}-\frac{1}{d-1}
\left[{\rm e}^{-\lambda t (d-1)}-1+\lambda t (d-1)\right]
\right.
\nonumber
\\
&& \hspace*{0.1cm}
\left.
-\sum_{n=3}^{\infty} \frac{(-\lambda t)^n}{n!}
\sum_{k=1}^{n-2} (d-1)^k
\sum_{j=1}^{n-k-1} {k+1 \choose j} Q_{j,j+n-k-1}
\right\}.
\label{solp0LneqM}
\end{eqnarray}
\end{proposition}
\begin{proof}
The proof proceeds similarly as that of Proposition \ref{prop1}. Recalling Eq.\ (\ref{eq_def1menoGt}),
for $0<t<\log(\mu/\lambda)/(\mu-\lambda)$ we have
\begin{equation}
1-G(t)=\frac{\frac{\mu}{\lambda}-1}{\frac{\mu}{\lambda}-{\rm e}^{-\lambda t \left(\frac{\mu}{\lambda}-1\right)}}
=\sum_{n=0}^{\infty} A_n\left(\frac{\mu}{\lambda}\right) \frac{(-\lambda t)^n}{n!}.
\label{eq_serie}
\end{equation}
The radius of convergence of the power series in Eq.\ (\ref{eq_serie}) has been determined finding the
location (in the complex plane) of the singularity nearest to the origin.
We can construct the following homotopy
\begin{eqnarray}
&& \hspace*{-1.6cm}
H(p,q)=\sum_{n=0}^{+\infty} q_n(t) q^n-\sum_{n=0}^{+\infty} A_n\left(\frac{\mu}{\lambda}\right)
\frac{(-\lambda t)^n}{n!} q^n
\nonumber
\\
&& \hspace*{-0.2cm}
+\lambda (d-1)
\sum_{n=0}^{+\infty} q^{n+1} \sum_{j=0}^{n} A_{j+1}\left(\frac{\mu}{\lambda}\right)
\int_{0}^t \frac{(-\lambda y)^j}{j!} q_{n-j}(t-y) {\rm d}y=0,
\label{eq:homotopy3}
\end{eqnarray}
where $q$ is the embedding parameter and we have set $p(0,t)=\sum_{n=0}^{+\infty} q_n(t) q^n$.
We thus find that $q_n(t)$ satisfies the following recursive equation:
\begin{equation}
q_n(t)=A_n\left(\frac{\mu}{\lambda}\right)\frac{(-\lambda t)^n}{n!}
-\lambda (d-1) \sum_{j=0}^{n-1} A_{j+1}\left(\frac{\mu}{\lambda}\right)
\int_{0}^{t} \frac{(-\lambda y)^j}{j!} q_{n-1-j}(t-y) {\rm d}y.
\label{eq:qnrecursive2}
\end{equation}
By straightforward calculations, from Eq.\ (\ref{eq:qnrecursive2}) one immediately gets
\begin{equation}
q_0(t)=1,\qquad q_1(t)=-d\lambda  t,\qquad q_2(t)=\lambda d (\mu+\lambda d)\frac{t^2}{2}.
\label{eq:q0andq1}
\end{equation}
Let us now make use of the strong induction principle to show that, for $n\geq 3$, 
\begin{equation}
 q_n(t)=d \,\frac{(-\lambda t)^n}{n!}\bigg\{ \sum_{k=1}^{n-2} (d-1)^k
 \sum_{j=1}^{n-k-1} {k+1 \choose j} Q_{j,j+n-k-1}  
 +A_{n}\left(\frac{\mu}{\lambda}\right)+(d-1)^{n-1}\bigg\}.
\label{eq:qnLdiffM}
\end{equation}
%
By direct calculations, it follows from 
(\ref{eq:qnrecursive2}) and (\ref{eq:q0andq1}) that
$$
q_3(t)=-\lambda d[\lambda^2 d^2+2\lambda \mu(d+1)+\mu^2]\frac{t^3}{3!},
$$
which is equal to Eq. (\ref{eq:qnLdiffM}) for $n=3$, being $A_3(z)=1+4 z+z^2$.
Let us consider $n\geq 3$ and assume that Eq.\ (\ref{eq:qnLdiffM}) holds for all $r=2,\ldots,n-1$. We shall
prove that identity (\ref{eq:qnLdiffM}) holds also for $r=n$.
From Eq.\ (\ref{eq:qnrecursive2}), recalling (\ref{eq:q0andq1}) we have
\begin{eqnarray*}
&& \hspace*{-2cm}
q_n(t)=d A_n\left(\frac{\mu}{\lambda}\right)\frac{(-\lambda t)^n}{n!}
+d (d-1) A_{n-1}\left(\frac{\mu}{\lambda}\right)\frac{(-\lambda t)^n}{n!}
\nonumber
\\
&& \hspace*{-1cm}
-\lambda (d-1) \sum_{r=2}^{n-1} A_{n-r}\left(\frac{\mu}{\lambda}\right)
\int_{0}^{t} \frac{(-\lambda y)^{n-1-r}}{(n-1-r)!} q_{r}(t-y) {\rm d}y.
\end{eqnarray*}
Hence, due to the induction hypothesis (\ref{eq:qnLdiffM}), we obtain
\begin{eqnarray}
\label{eq:qn4}
&& \hspace*{-1.cm}
q_n(t)=d \frac{(-\lambda t)^n}{n!}\bigg\{ A_n\left(\frac{\mu}{\lambda}\right)
+ (d-1) A_{n-1}\left(\frac{\mu}{\lambda}\right)
\\
&& \hspace*{-0.1cm}
+(d-1) \sum_{r=2}^{n-1} A_{n-r}\left(\frac{\mu}{\lambda}\right)A_{r}\left(\frac{\mu}{\lambda}\right)
+\sum_{r=2}^{n-1} (d-1)^{r} A_{n-r}\left(\frac{\mu}{\lambda}\right)
\nonumber
\\
&& \hspace*{-0.1cm}
+ \sum_{r=3}^{n-1} A_{n-r}\left(\frac{\mu}{\lambda}\right)
\sum_{k=1}^{r-2} (d-1)^{k+1}
\sum_{j=1}^{r-k-1} {k+1 \choose j}  Q_{j,j+r-k-1} \bigg\}
\nonumber
\\
&& \hspace*{-0.1cm}
=d \frac{(-\lambda t)^n}{n!}\left\{ A_n\left(\frac{\mu}{\lambda}\right)
+\sum_{r=1}^{n-1} (d-1)^r A_{n-r}\left(\frac{\mu}{\lambda}\right)
\right.
\nonumber
\\
&& \hspace*{-0.1cm}
\left.
+ \sum_{h=1}^{n-2} (d-1)^h
\sum_{j=1}^{n-1-h} {h \choose j}
\sum_{r=1}^{n-j-h} A_{r}\left(\frac{\mu}{\lambda}\right) Q_{j,n+j-h-r} \right\}.
\nonumber
\end{eqnarray}
Noting that
$$
 \sum_{r=1}^{n-j-h} A_{r}\left(\frac{\mu}{\lambda}\right) Q_{j,n+j-h-r}
 = Q_{j,n+j-h-1}+ Q_{{j+1},n+j-h},
$$
from Eq.\ (\ref{eq:qn4}) we obtain
\begin{eqnarray}
&& \hspace*{-0.4cm}
q_n(t)=d \frac{(-\lambda t)^n}{n!}\left\{ A_n\left(\frac{\mu}{\lambda}\right)
+\sum_{r=1}^{n-1} (d-1)^{n-r} A_{r}\left(\frac{\mu}{\lambda}\right)
\right.
\nonumber
\\
&& \hspace*{0.cm}
 + \sum_{r=2}^{n-1} (d-1)^{n-r}
 \sum_{k=1}^{r-1} {n-r \choose k} Q_{k,r+k-1}
\left. 
 + \sum_{r=2}^{n-1} (d-1)^{n-r}
 \sum_{k=2}^{r} {n-r \choose k-1} Q_{k,r+k-1}
\right\}
\nonumber
\\
&& \hspace*{-0.2cm}
=d \frac{(-\lambda t)^n}{n!}\left\{ A_n\left(\frac{\mu}{\lambda}\right)
+\sum_{r=1}^{n-1} (d-1)^{n-r} A_{r}\left(\frac{\mu}{\lambda}\right)
\right.
\nonumber
\\
&& \hspace*{0.cm}
+ \sum_{r=2}^{n-1} (d-1)^{n-r} (n-r) A_{r}\left(\frac{\mu}{\lambda}\right)
\left.
+ \sum_{r=2}^{n-1} (d-1)^{n-r}
\sum_{k=2}^{r-1} {n-r+1 \choose k} Q_{k,r+k-1}
\right\}
\nonumber
\\
&& \hspace*{-0.2cm}
=d \frac{(-\lambda t)^n}{n!}\left\{ A_n\left(\frac{\mu}{\lambda}\right)+(d-1)^{n-1}
+ \sum_{r=2}^{n-1} (d-1)^{n-r}  \sum_{k=1}^{r-1} {n-r+1 \choose k} Q_{k,r+k-1}
\right\},
\nonumber
\end{eqnarray}
which gives Eq.\ (\ref{eq:qnLdiffM}). By setting $q=1$ in assumption
$p(0,t)=\sum_{n=0}^{+\infty} q_n(t) q^n$, and recalling (\ref{eq:qnLdiffM}), we finally obtain
Eq.\ \eqref{solp0LneqM}.
\end{proof}
\begin{remark}\rm
If $d=1$ the expressions for $p(0,t)$ given in Theorem \ref{theorem1}, Proposition \ref{prop1} and
Proposition \ref{p0LneqM} are in agreement with the well-known results for the linear birth-death process
with immigration (see, for instance, Section 2.3 of Nucho \cite{Nu81}).
\end{remark}
\par
Hereafter we derive an explicit expression for $t_{n,k}$ in terms of multinomial coefficients,
for $n\geq 2$ and $1\leq k\leq n$. It is worth pointing out that a closed form expression
for such numbers does not appear to have been obtained before.
\begin{corollary}
The following equalities hold for $n\geq 2$:
\begin{eqnarray*}
&& \hspace*{-0.7cm}
t_{n,1}=n!+(-1)^{n-1}+\sum_{k=1}^{n-2} (-1)^k
\sum_{j=1}^{n-k-1} {k+1 \choose j} (n-k-1+j)!
\sum_{\substack{s_1,\ldots,s_j\geq 2\\s_1+\cdots+s_j=n-k-1+j}}
\frac{1}{{n-k-1+j\choose s_{1},\ldots,s_{j}}};
\\
&& \hspace*{-0.7cm}
t_{n,k}={n-1\choose k-1}(-1)^{n-k}+\sum_{r=k-1}^{n-2}{r \choose k-1} (-1)^{r-k+1}
\sum_{j=1}^{n-r-1} {r+1 \choose j} 
\\
&& \hspace*{0.1cm}
\times (n-r-1+j)!\sum_{\substack{s_1,\ldots,s_j\geq 2\\s_1+\cdots+s_j=n-r-1+j}}
\frac{1}{{n-r-1+j\choose s_{1},\ldots,s_{j}}},
\hspace{2.5cm} 2\leq k\leq n-1;
\\
&& \hspace*{-0.7cm}
t_{n_,n}=1.
\end{eqnarray*}
\end{corollary}
\begin{proof}
The proof follows from Propositions \ref{prop1} and \ref{p0LneqM}, by letting $\mu \to \lambda$
and noting that $A_j(1)=j!$.
\end{proof}
\par
In the following proposition we obtain the probability generating function when $\mu\neq \lambda$
and $\alpha=\lambda$. In the sequel we shall denote by
\begin{equation}
 \theta_n\equiv \theta_n(\lambda,\mu,d):=\sum_{i=1}^{n-2} (d-1)^i
\sum_{j=1}^{n-i-1} {i+1 \choose j}
Q_{j,j+n-i-1} + A_{n}\left(\frac{\mu}{\lambda}\right)+(d-1)^{n-1},
\label{coeff_teta}
\end{equation}
where $Q_{j,m}$ is defined in Eq. (\ref{eq:defsum}).
\begin{proposition}\label{prop:mumaggla}
If $\alpha=\lambda$ and $\mu\neq \lambda$, for $t<\log(\mu/\lambda)/(\mu-\lambda)$, it is
\begin{eqnarray}
&& \hspace*{-2cm}
F(z,t)=1-d+\frac{d(\mu-\lambda)}{\mu-\lambda z-\lambda(1-z){\rm e}^{-(\mu-\lambda) t}}
-\frac{d (d-1)(\mu-\lambda)^2}{\mu-\lambda z}
\nonumber
\\
&& \hspace*{-0.7cm}
\times \Big[-\lambda g_1(z)+ \sum_{n=2}^{+\infty} \frac{(-\lambda)^n}{n!}g_n(z)\theta_n  \Big]
\label{funzioneF_MmagL}
\end{eqnarray}
where
$$
  g_k(z)=\left\{
  \begin{array}{l}
  \frac{1}{(\lambda-\mu)^{k+1}}\left\{\lambda (1-z) (\lambda-\mu)^{k-1} t^k+
 k!\left[{\rm Li}_{k}\left(\frac{\lambda(1-z){\rm e}^{-(\mu-\lambda) t}}{\mu-\lambda z}\right)
 \right.\right.  \\
  \left.\left.
 \qquad\qquad -{\rm Li}_{k}\left(\frac{\lambda(1-z)}{\mu-\lambda z}\right) \right]
 -\sum_{r=1}^{k-1} \frac{k!}{r!}[(\lambda-\mu)t]^r
 {\rm Li}_{k-r}\left(\frac{\lambda(1-z)}{\mu-\lambda z}\right)\right\}   \\
 \hspace{8cm} \hbox{if }\mu>\lambda, \\[0.4cm]
  \frac{1}{(\mu-\lambda)^{k+1}}\left\{(\mu-\lambda z) (\mu-\lambda)^{k-1} t^k+
 k!\left[{\rm Li}_{k}\left(\frac{(\mu-\lambda z){\rm e}^{-(\lambda-\mu) t}}{\lambda (1-z)}\right)
 \right. \right.  \\
 \qquad\qquad
 \left.\left. -{\rm Li}_{k}\left(\frac{\mu-\lambda z}{\lambda(1-z)}\right) \right]
-\sum_{r=1}^{k-1} \frac{k!}{r!}\,[(\mu-\lambda)t]^r \,{\rm Li}_{k-r}\left(\frac{\mu-\lambda z}{\lambda(1-z)}\right)
\right\} \\ 
 \hspace{8cm} \hbox{if }\mu<\lambda,
 \end{array}
 \right.
$$
and where
\begin{equation}
 {\rm Li}_{k}(z)=\sum_{j=1}^{+\infty} \frac{z^j}{j^k}
 \label{eq:defLi}
\end{equation}
is the polylogarithm function.
\end{proposition}
\begin{proof}
It immediately follows from Eqs. (\ref{soluzioneFgenerale}) and (\ref{funzioneh}),
recalling Eq. (\ref{solp0LneqM}).
\end{proof}
\par
We conclude this section by evaluating the probability (\ref{pgrandekt}).
\begin{proposition}\label{prop:Pkt}
Let $k\in\mathbb{N}^+$.
If $\alpha=\lambda$, for $t<\log(\mu/\lambda)/(\mu-\lambda)$, we have\\
$\bullet$ for $\mu>\lambda$
\begin{eqnarray}
&& \hspace*{-0.5cm}
P(k,t)
=\frac{\lambda^k}{\mu^{k+1}}\left\{d (\mu-\lambda)\mu^{k+1} \frac{[1-{\rm e}^{-(\mu-\lambda)t}]^k}
{[\mu-\lambda {\rm e}^{-(\mu-\lambda)t}]^{k+1}}-d (d-1)(\lambda-\mu)
\right.
\label{probunione_MmagL}
\\
&& \hspace*{0.6cm}
\left.
\times\left[-\lambda t +\sum_{n=2}^{+\infty} \frac{(-\lambda t)^n}{n!}\theta_n\right]
-\lambda d (d-1) \sum_{l=1}^{+\infty} \frac{(1-{\rm e}^{-l(\mu-\lambda)t})}{l}
{}_2F_1^*
\right.
\nonumber
\\
&& \hspace*{0.6cm}
\left. +d(d-1)\sum_{n=2}^{+\infty}\frac{(-\lambda)^n}{(\lambda-\mu)^{n-1}}\theta_n
\sum_{r=1}^{n-1}\frac{[(\lambda-\mu)t]^r}{r!}
\sum_{l=1}^{+\infty}\frac{1}{l^{n-r}}{}_2F_1^*
\right.
\nonumber
\\
&& \hspace*{0.6cm}
\left.
+d(d-1)\sum_{n=2}^{+\infty}\frac{(-\lambda)^n}{(\lambda-\mu)^{n-1}}\theta_n
\sum_{l=1}^{+\infty}\frac{1-{\rm e}^{-l (\mu-\lambda) t}}{l^{n}}{}_2F_1^*
\right\},
\nonumber
\end{eqnarray}
where ${}_2F_1^*={}_2F_1\left(-l,k+1;1;1-\frac{\lambda}{\mu}\right)$; \\
$\bullet$ for $\mu<\lambda$
\begin{eqnarray}
&& \hspace*{-0.3cm}
P(k,t)
=d (\lambda-\mu){\rm e}^{-(\lambda-\mu)t} \frac{[\lambda\,(1-{\rm e}^{-(\lambda-\mu)t})]^k}
{[\lambda-\mu {\rm e}^{-(\lambda-\mu)t}]^{k+1}}-d (d-1) \frac{\lambda}{\mu}\,{\rm Li}_{k+1}\left(\frac{\mu}{\lambda}\right)
\label{probunione_LmagM}
\\
&& \hspace*{0.9cm}
+ d(d-1) \sum_{s=1}^{+\infty} \frac{{\rm e}^{-s(\lambda-\mu)t}}{s}
{}_2F_1^{**}
\nonumber
\\
&& \hspace*{0.9cm}
+\frac{d(d-1)}{\lambda}\sum_{n=2}^{+\infty}\frac{(-\lambda)^n}{(\mu-\lambda)^{n-1}}\theta_n
\sum_{s=1}^{+\infty}\frac{(1-{\rm e}^{-s (\lambda-\mu) t})}{s^n}
{}_2F_1^{**}
\nonumber
\\
&& \hspace*{0.9cm}
+\frac{d(d-1)}{\lambda}\sum_{n=2}^{+\infty}\frac{(-\lambda)^n}{(\mu-\lambda)^{n-1}}\theta_n
\sum_{r=1}^{n-1}\frac{[(\mu-\lambda)t]^r}{r!}
\sum_{s=1}^{+\infty}\frac{1}{s^{n-r}}
{}_2F_1^{**},
\nonumber
\end{eqnarray}
where $\theta_n$ is defined in Eq. (\ref{coeff_teta}), and ${}_2F_1^{**}={}_2F_1\left(1-s,k+1;1;1-\frac{\mu}{\lambda}\right)$.
\end{proposition}
\begin{proof}
The proof follows from Eqs. (\ref{FPgrande}) and (\ref{funzioneF_MmagL}), recalling that $p(0,t)$ satisfies
the integral equation (\ref{eq:eqpot}), noting  that (see Eq. (65.1.3) of \cite{Hansen1975}, for instance)
$$
 \sum_{k=0}^{+\infty} \frac{(d)_k}{k!} y^k {}_2F_1\left(-k,b;c;x\right)
 =(1-y)^{-d}{}_2F_1\left(d,b;c;\frac{x y}{y-1}\right),
$$
and making use of (\ref{eq:defLi}).
\end{proof}
%
%
\begin{figure}[h] 
\begin{center}
\epsfxsize=6.5cm
\epsfbox{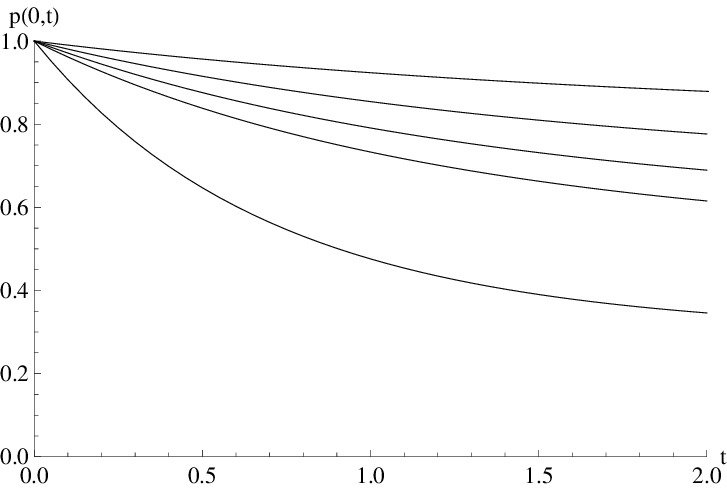}
$\;$
\epsfxsize=6.5cm
\epsfbox{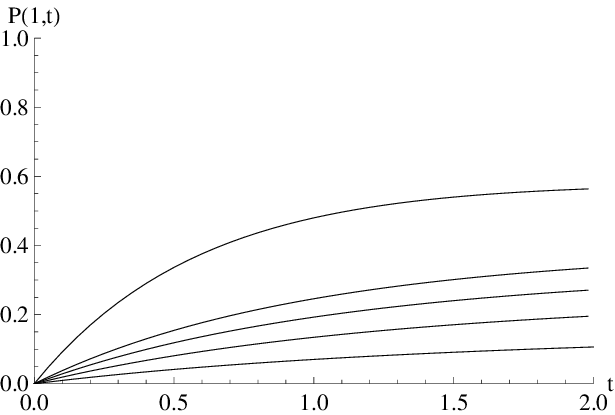}
\caption{Plots of $p(0,t)$ and $P(1,t)$ for $\lambda=0.1$, $\mu=0.5$ and $\alpha=0.1$,
for $d=1,2,3,4,10$, from top to bottom for $p(0,t)$, and from bottom to top for $P(1,t)$.}
\label{fig:LminM}
\end{center}
\end{figure}
\begin{figure}[h] 
\begin{center}
\epsfxsize=6.5cm
\epsfbox{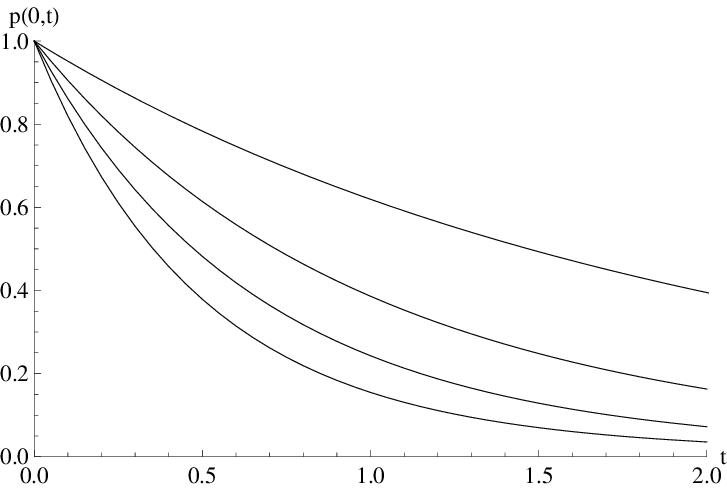}
$\;$
\epsfxsize=6.5cm
\epsfbox{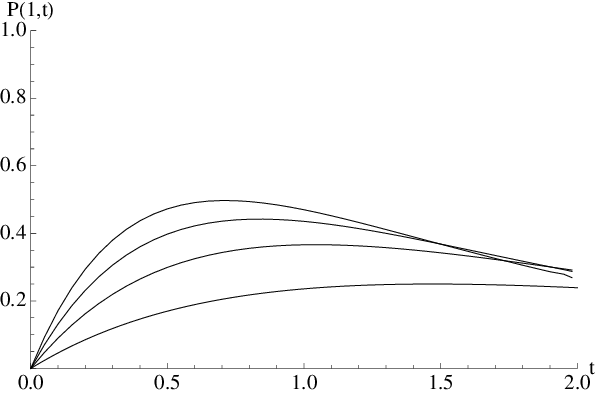}
\caption{Plots of $p(0,t)$ and $P(1,t)$ for $\lambda=0.5$, $\mu=0.1$ and $\alpha=0.5$,
for $d=1,2,3,4$, from top to bottom for $p(0,t)$, and from bottom to top for $P(1,t)$.}
\label{fig:LmagM}
\end{center}
\end{figure}
%
\par
In Figures \ref{fig:LminM} and \ref{fig:LmagM} we show some plots of $p(0,t)$ and $P(1,t)$ obtained
by evaluating the expressions given in Proposition \ref{p0LneqM} and Proposition \ref{prop:Pkt}.
\par
Let us now provide a simple relation between the polylogarithm function and a series of Gauss
hypergeometric functions, which does not appear to have been given before.
This result immediately follows from Proposition \ref{prop:Pkt}.
\begin{corollary}
For all $k\in\mathbb{N}^+$ and $x\in (0,1)$ we have
\begin{equation}
 {\rm Li}_{k+1}(x)=x\sum_{s=1}^{+\infty} \frac{1}{s} \left[ {}_2F_1\left(1-s,k+1;1;1-x\right)\right].
 \label{eq:serie}
\end{equation}
\end{corollary}
\begin{proof}
The proof of (\ref{eq:serie}) follows from (\ref{probunione_LmagM}), by taking into account that
$P(k,0)=0$.
\end{proof}
\par
A classical problem in population birth-death models is the
extinction, i.e.\ the first passage through the zero state (see, e.g.\ 
Van Doorn and Zeifman \cite{VaDoZe2005}). However,  in our model this problem  
reduces to a well-known one-dimensional case. 
\par
We finally conclude the analysis of $\cal N$ by discussing some asymptotic results.
\section{Asymptotic results}\label{section:4}
According to the one-dimensional case, ${\cal N}$ admits a stationary distribution 
if and only if $\lambda<\mu$. In the following proposition we obtain explicitly the 
expression of the stationary probabilities 
\begin{equation}
 \rho(0):=\lim_{t\rightarrow +\infty} p(0,t),
 \qquad 
 \rho(k):=\lim_{t\rightarrow +\infty} P(k,t),
 \quad k\in\mathbb{N}^+.
\nonumber
\end{equation}
\begin{proposition}\label{stationaryprob}
If $\lambda<\mu$, then
\begin{equation}
\rho(0)
 = \frac{1}{d} \,\displaystyle{ \frac{\left(1-\frac{\lambda}{\mu}\right)^{\frac{\alpha}{\lambda}}}
 {1-\left(1-\frac{1}{d}\right)\left(1-\frac{\lambda}{\mu}\right)^{\frac{\alpha}{\lambda}}}},
\label{p0limit}
\end{equation}
\begin{equation}
\rho(k) 
=\displaystyle{\frac{\left(1-\frac{\lambda}{\mu}\right)^{\frac{\alpha}{\lambda}}}
{1-\left(1-\frac{1}{d}\right)\left(1-\frac{\lambda}{\mu}\right)^{\frac{\alpha}{\lambda}}}
\,\frac{(\frac{\alpha}{\lambda})_k}{k!}\,\left(\frac{\lambda}{\mu}\right)^k },
\qquad k\in\mathbb{N}^+.
\label{asymptPgrande}
\end{equation}
If $\lambda\geq \mu$, then $\rho(k)=0$ for $k\in\mathbb{N}_0$.
\end{proposition}
\begin{proof}
Eq.\ (\ref{p0limit}) follows from Theorem \ref{theorem1}. Denoting by
\begin{equation}
 {\cal L}_s[f(t)]=\int_{0}^{+\infty} {\rm e}^{-s t} f(t) {\rm d}t,\qquad s\geq 0,
 \label{eq:defLsft}
\end{equation}
the Laplace transform of an arbitrary function $f(t)$, from Proposition \ref{propFzt} we have
\begin{equation}
 {\cal L}_s[F(z,t)]= {\cal L}_s[H(t)]+(d-1)(s\,{\cal L}_s[H(t)]-1)\, {\cal L}_s[p(0,t)].
\label{taub1}
\end{equation}
Note that, due to Eq.\ (10) of Section 2.1.3, p.\ 59, of Erd\'elyi {\em et al.}\ \cite{ErMaObTr53},
$$
 {\cal L}_s[H(t)]=\left\{
 \begin{array}{ll}
 \displaystyle{\frac{\lambda (\lambda-\mu)^{\frac{\alpha}{\lambda}}
{}_{2}F_{1}(\frac{\alpha}{\lambda},\frac{\alpha}{\lambda}+\frac{s}{\lambda-\mu};\frac{\alpha}{\lambda}+\frac{s}{\lambda-\mu}+1;\frac{\mu-\lambda z}{\lambda(1-z)})}
{\left[\lambda(1-z)\right]^{\frac{\alpha}{\lambda}} \left[\lambda s+\alpha (\lambda-\mu)\right]}}, & \lambda> \mu,
\\[0.4cm]
\displaystyle{\frac{(\mu-\lambda)^{{\frac{\alpha}{\lambda}}-1}}{(\mu-\lambda z)^{\frac{\alpha}{\lambda}}}
\frac{\mu-\lambda}{s} {}_{2}F_{1}\left(\frac{\alpha}{\lambda},\frac{s}{\mu-\lambda};\frac{s}{\mu-\lambda}+1;\frac{\lambda(1-z)}{\mu-\lambda z}\right)},
& \lambda<\mu,
\\[0.4cm]
\displaystyle{\frac{{\rm e}^{\frac{s}{\lambda(1-z)}}E\left(\frac{\alpha}{\lambda},\frac{s}{\lambda(1-z)}\right) }{\lambda(1-z)}},
& \lambda=\mu,
 \end{array}
 \right.
$$
where
\begin{equation}
E(\nu,z):=\int_{1}^{+\infty} \frac{{\rm e}^{-z t}}{t^\nu} {\rm d}t,\quad \nu\in {\mathbb R},\quad z>0,
\label{expintfunction}
\end{equation}
denotes the generalized exponential integral function and ${}_{2}F_{1}$ is defined in
(\ref{Gaussipergeomfunction}). Hence, recalling the 
Tauberian theorem, and making use of Eqs.\ (\ref{p0limit}) and (\ref{taub1}), we have
\begin{equation}
 \lim_{t\to +\infty} F(z,t)=\left\{
 \begin{array}{ll}
 0, & \lambda\geq  \mu,
\\[0.2cm]
\displaystyle{\frac{(\mu-\lambda)^{\frac{\alpha}{\lambda}}}{(\mu-\lambda z)^{\frac{\alpha}{\lambda}}}\left[1+\frac{(d-1) \left(\frac{\mu-\lambda}{\mu}\right)^{\frac{\alpha}{\lambda}}}{d-(d-1)\left(\frac{\mu-\lambda}{\mu}\right)^{\frac{\alpha}{\lambda}}}\right]} 
\displaystyle -\frac{(d-1) \left(\frac{\mu-\lambda}{\mu}\right)^{\frac{\alpha}{\lambda}}}{d-(d-1)\left(\frac{\mu-\lambda}{\mu}\right)^{\frac{\alpha}{\lambda}}},
& \lambda<\mu.
 \end{array}
 \right.
\label{Fztlimit}
\end{equation}
If $\lambda<\mu$, making use of
$$
\frac{(\mu-\lambda)^{\frac{\alpha}{\lambda}}}{(\mu-\lambda z)^{\frac{\alpha}{\lambda}}}
=\sum_{k=0}^{+\infty}\frac{(\frac{\alpha}{\lambda})_k}{k!}
\left(1-\frac{\lambda}{\mu}\right)^{\frac{\alpha}{\lambda}}
\left(\frac{\lambda}{\mu}\right)^k z^k
$$
and recalling Eq.\ (\ref{FPgrande}), after some calculations we obtain (\ref{asymptPgrande})
from Eqs.\ (\ref{p0limit}) and (\ref{Fztlimit}).
\end{proof}
\begin{remark}\rm
From the stationary probabilities (\ref{p0limit}) and (\ref{asymptPgrande}) we have 
that, for $\lambda<\mu$, the following identity holds for $k\in\mathbb{N}_0$:
\begin{equation}
 \rho(k)=\vartheta_d\,\pi(k)+(1-\vartheta_d)\,{\bf 1}_{\{k=0\}},
 \qquad 
 \left(\vartheta_d=
 \left[ {1-\left(1-\frac{1}{d}\right)\left(1-\frac{\lambda}{\mu}\right)^{\frac{\alpha}{\lambda}}}\right]^{-1}\right),
\end{equation}
where $\pi(k)$ is the negative binomial distribution given by 
$$
 \pi(k)= \left(1-\frac{\lambda}{\mu}\right)^{\frac{\alpha}{\lambda}}
 \frac{(\frac{\alpha}{\lambda})_k}{k!}\,\left(\frac{\lambda}{\mu}\right)^k,
\qquad  k\in\mathbb{N}_0.
$$
We note that if $d=1$ then $\vartheta_d=1$, and thus $\rho(k)=\pi(k)$, $\forall k\in\mathbb{N}_0$. 
Moreover, $\vartheta_d$ is increasing in $d\geq 1$ and tends to a constant when $d\to +\infty$.  
\end{remark}
\par
Denoting by $N$ the discrete random variable having distribution $\{\rho(k);\; k\in \mathbb{N}_0\}$,  
after some calculations we obtain the following mean and variance, for $\lambda<\mu$:
\begin{equation}
 E[N]=\vartheta_d\,\displaystyle\frac{\alpha}{\mu-\lambda}, \qquad 
 Var[N]=\vartheta_d^2\, \displaystyle\frac{\alpha \mu}{(\mu-\lambda)^2}\left[1-\left(1-\frac{1}{d}\right)
 \left(\frac{\alpha}{\mu}+1\right)\left(1-\frac{\lambda}{\mu}\right)^{\frac{\alpha}{\lambda}}\right].
\label{eq:EVarN}
\end{equation}
%
\section{The diffusion approximation}\label{section:5}
In this section we construct a diffusion approximation for the process $\cal N$.
We adopt a scaling procedure that is customary in queueing theory contexts (see, for
instance, Di Crescenzo {\em et al.} \cite{DiGiNoRi03}). First of all, we perform a different
parameterization of the model studied in Section \ref{Section:2} by setting
\begin{equation}
 \alpha=\tilde \gamma\, {\tilde \mu\over \epsilon},
 \qquad
 \lambda=  {\tilde \mu\over \epsilon} + \tilde \beta,
 \qquad
 \mu=  {\tilde \mu\over \epsilon},
 \label{eq:parametri}
\end{equation}
with $\tilde\gamma>0$, $\tilde \mu>0$, $\tilde \beta\in\mathbb{R}$ and $\epsilon>0$. Note that
$\epsilon$ is a positive constant that can be viewed as a measure of the size of $\tilde \mu$. It
plays a crucial role in the approximating procedure indicated below, where   $\epsilon\to 0^+$.
\par
For all $t>0$, consider the scaling $N^*_{\epsilon}(t)=N(t)\,\epsilon$, so that 
${\cal N}^*_{\epsilon}:=\{(N^*_{\epsilon}(t), J(t));\;t\geq 0\}$ 
is a continuous-time stochastic process having state space
$S_{\epsilon}^*=\{0\}\cup\left(\mathbb{N}^+_{\epsilon}\times D\right)$, 
where $\mathbb{N}^+_{\epsilon}=\{\epsilon,2 \epsilon,3 \epsilon,\ldots\}$. 
The transient probabilities, for $t\geq 0$, $k\in\mathbb{N}_0$, $j\in D$, are given by 
\begin{eqnarray*}
 p^*_{\epsilon}(0,t) \!\!  &:=& \!\!  {\mathbb P}\left\{(N^*_{\epsilon}(t),J(t))=0\right\}, \\
p^*_{\epsilon}(k,j,t) \!\!  &:=& \!\!  {\mathbb P}\left\{(N^*_{\epsilon}(t),J(t))=(k \epsilon,j)\right\}  
 = {\mathbb P}\left\{k\epsilon\leq N^*_{\epsilon}(t)<(k+1)\epsilon, J(t)=j\right\}.
\end{eqnarray*}
In the limit as $\epsilon\to 0^+$, the scaled process ${\cal N}^*_{\epsilon}$ is shown to converge
weakly to a diffusion process ${\cal X}:=\{(X(t),J(t));\;t\geq 0\}$, whose state space is the star graph
$S_{\cal X}:=\{0\}\cup\left(\mathbb{R}^+\times D\right)$. 
For $x\in \mathbb{R}^+$, $t\geq 0$ and $j\in D$, let
${\mathbb P}\{x\leq X(t)<x+\epsilon, J(t)=j\}=f(x,j,t)\epsilon+o(\epsilon)$, so that $f(x,j,t)$ denotes
the density of the process in state $x$  on the ray $S_j$. 
\begin{proposition}\label{prop:diffj}
For $x\in \mathbb{R}^+$, $t> 0$ and $j\in D$, the following differential equation holds:
\begin{equation}
 {\partial\over\partial t}\;f(x,j,t)=-{\partial\over\partial x}\;
 \Bigl\{(\tilde \beta\,x+\tilde \gamma\, \tilde \mu)\,
 f(x,j,t)\Bigr\}
 +{1\over 2}\,{\partial^2\over\partial x^2}\Bigl\{2\,\tilde \mu\, x\,f(x,j,t)\Bigr\},
 \label{eq:equdiff}
\end{equation}
with boundary condition
\begin{equation}
  \sum_{j=1}^d \lim_{x\to 0^+}\Bigl\{(\tilde \beta\,x+\tilde \gamma\, \tilde \mu)\,f(x,j,t)
  -{1\over 2}\,{\partial\over\partial x}\bigl[2\,\tilde \mu\, x\,f(x,j,t)\bigr]\Bigr\}=0.
  \label{eq:rifless}
\end{equation}
\end{proposition}
\begin{proof}
Since $p^*_{\epsilon}(k,j,t)=p(k,j,t)$, due to (\ref{eq:parametri}) and in analogy with system (\ref{eq:system}),
for $j\in D$ and $t> 0$ we have
\begin{eqnarray}
&& \hspace{-1cm}
 p^*_{\epsilon}(0, t+\Delta t)  =  \sum_{j=1}^d p^*_{\epsilon}(1,j,t)\,{\tilde \mu\over \epsilon}\Delta t
 +p^*_{\epsilon}(0,t)\,\left(1-d\tilde \gamma\, {\tilde \mu\over \epsilon} \Delta t\right)+o(\Delta t),
 \label{eq:nsystem1} \\
&& \hspace{-1cm}
 p^*_{\epsilon}(k,j, t+\Delta t) = p^*_{\epsilon}(k-1,j,t)\,
 \left[\tilde \gamma\, {\tilde \mu\over \epsilon}
 +\left({\tilde \mu\over \epsilon} + \tilde \beta\right)(k-1)\right]\Delta t
 +  p^*_{\epsilon}(k+1,j,t)\,{\tilde \mu\over \epsilon}\,(k+1) \Delta t
 \nonumber \\
 &&  \hspace{1.4cm}  + p^*_{\epsilon}(k,j,t)\,\left\{1-\left[\tilde \gamma\, {\tilde \mu\over \epsilon}
 +\left(2{\tilde \mu\over \epsilon} + \tilde \beta\right)k\right]\Delta t\right\}
 +o(\Delta t),
 \qquad  k\in\mathbb{N}^+.
   \label{eq:nsystem2}
\end{eqnarray}
Let $p^*_{\epsilon}(k,j,t)\simeq f(k\epsilon,j,t)\, \epsilon$ for $\epsilon$ close to 0.
Hence, for $x=k \epsilon$, from Eq.\ (\ref{eq:nsystem2}) we have
\begin{eqnarray*}
f(x,j,t+\Delta t) \!\!  &=& \!\!  f(x-\epsilon,j,t)\,\left[\tilde \gamma\, \tilde \mu
 +\left({\tilde \mu\over \epsilon} + \tilde \beta\right)(x-\epsilon)\right]{\Delta t\over \epsilon} 
 +  f(x+\epsilon,j,t)\,\tilde \mu\,(x+\epsilon) {\Delta t\over \epsilon}
  \\
 & + & \!\! f(x,j,t)\,\left\{1-\left[\tilde \gamma\, \tilde \mu
 +\left(2{\tilde \mu\over \epsilon} + \tilde \beta\right)x\right]{\Delta t\over \epsilon}\right\}
 +o(\Delta t).
\end{eqnarray*}
Expanding $f$ as Taylor series, by setting $\Delta t=A\, \epsilon ^2$, with $A>0$, and passing to
the limit as $\epsilon\to 0^+$, we obtain Eq.\ (\ref{eq:equdiff}). Similarly, Eq.\ (\ref{eq:nsystem1}) yields
$$
 f(0,t+\Delta t)
 =\sum_{j=1}^d f(\epsilon,j,t)\,\tilde \mu\,{\Delta t\over \epsilon}
 +f(0,t)\left(1 -d \tilde \gamma \tilde \mu\,{\Delta t \over \epsilon}\right)
 +o(\Delta t),
$$
so that (\ref{eq:rifless}) holds.
\end{proof}
\par
From the above procedure, the following approximation holds:
${\mathbb P}\{N(t)<k\}\simeq {\mathbb P}\{X(t)<k\epsilon\}$,
this being expected to improve as $\epsilon \to 0^+$ and  $k\to +\infty$.
\par
Let us now introduce the density
\begin{equation}
 h(x,t):=\sum_{j=1}^d f(x,j,t),\qquad x\in \mathbb{R}^+,\quad t\geq 0.
  \label{eq:definhxt}
\end{equation}
\begin{proposition}
For $x\in \mathbb{R}^+$ and $t> 0$, the transition density (\ref{eq:definhxt}) satisfies the following
differential equation:
\begin{equation}
 {\partial\over\partial t}\;h(x,t)=-{\partial\over\partial x}\;
 \Bigl\{(\tilde \beta\,x+\tilde \gamma\, \tilde \mu)\,
 h(x,t)\Bigr\}
 +{1\over 2}\,{\partial^2\over\partial x^2}\Bigl\{2\,\tilde \mu\, x\,h(x,t)\Bigr\},
 \label{eq:equdiffsomma}
\end{equation}
with boundary condition
\begin{equation}
  \lim_{x\to 0^+}\Bigl\{(\tilde \beta\,x+\tilde \gamma\, \tilde \mu)\,h(x,t)
  -{1\over 2}\,{\partial\over\partial x}\bigl[2\,\tilde \mu\, x\,h(x,t)\bigr]\Bigr\}=0
  \label{eq:equdiffbound}
\end{equation}
and Dirac-delta initial condition
\begin{equation}
 \lim_{t\to 0^+}h(x,t)=\delta(x).
  \label{eq:initcond}
\end{equation}
\end{proposition}
\begin{proof}
The proof of Eqs.\ (\ref{eq:equdiffsomma}) and (\ref{eq:equdiffbound})
follows immediately from Proposition \ref{prop:diffj}, and recalling position (\ref{eq:definhxt}).
The condition (\ref{eq:initcond}) can be obtained from (\ref{eq:probiniz}).
\end{proof}
\par
Note that Eq.\ $(\ref{eq:equdiffsomma})$ is the Fokker-Planck equation for a temporally
homogeneous diffusion process on $\mathbb{R}^+$ with linear drift and linear infinitesimal variance,
while Eq.\ $(\ref{eq:equdiffbound})$ expresses a zero-flux condition in the state $x=0$.
We remark that various results on such kind of diffusion process have been given in
Buonocore {\em et al.} \cite{BuCaNoPi2013}, Giorno {\em et al.} \cite{GiNoRiSa86},
and Sacerdote \cite{Sa90}, for instance.
\par
Hereafter we show that $h(x,t)$ is a gamma density with shape parameter $\tilde \gamma$ and rate $\psi(t)$.
\begin{proposition}
Let $\psi(t)=(\tilde \beta/\tilde \mu) ({\rm e}^{\tilde \beta t}-1)^{-1}$, $t> 0$. 
The density (\ref{eq:definhxt}) is given by
\begin{equation}
 h(x,t)=\frac{[\psi(t)]^{\tilde \gamma}}{\Gamma(\tilde \gamma)}\,x^{\tilde \gamma-1} \, {\rm e}^{-x \psi(t)},
 \qquad x\in\mathbb{R}^+, \quad t> 0.
 \label{eq:esprhxt}
\end{equation}
\end{proposition}
\begin{proof}
The transformation (see Capocelli and Ricciardi \cite{CapRic1976})
$$
 x'=x\, {\rm e}^{-\tilde \beta t},
 \qquad
 t'=\frac{\tilde \mu}{\tilde \beta}\left(1-{\rm e}^{-\tilde \beta t}\right),
 \qquad
 h(x,t)={\rm e}^{-\tilde \beta t} \,h'(x',t'),
$$
changes equation (\ref{eq:equdiffsomma})  and condition (\ref{eq:equdiffbound}) respectively into a
Fokker-Planck equation for the time-homogeneous diffusion process on $\mathbb{R}^+$ having drift
$\tilde \gamma$ and infinitesimal variance $2x'$, with a zero-flux condition on the boundary $x'=0$.
Initial condition (\ref{eq:initcond})  becomes $ \displaystyle\lim_{t'\to 0^+}h'(x',t')=\delta(x')$.
The proof thus proceeds similarly as Proposition 4.1 of Di Crescenzo and Nobile \cite{DiCrNo95}
assuming a zero initial state.
\end{proof}
\par
\begin{figure}[h] 
\begin{center}
\epsfxsize=6.5cm
\epsfbox{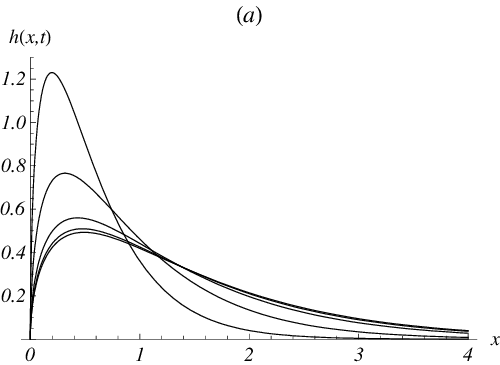}
$\;$
\epsfxsize=6.5cm
\epsfbox{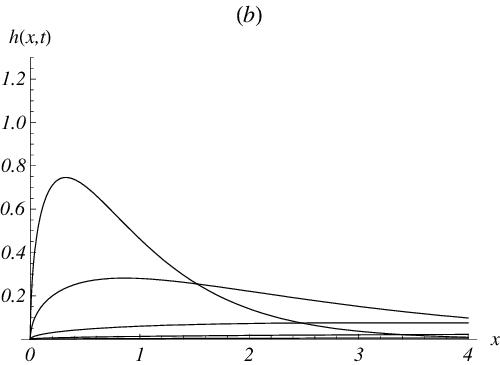}
\caption{Density (\ref{eq:esprhxt}) for $t=0.5, 1, 2, 3, 4$ (from top to bottom near the origin),
for $\tilde \mu=1$, $\tilde \gamma=1.5$ and (a) $\tilde \beta=-1$, (b)  $\tilde \beta=1$.
}
\label{fig:densh}
\end{center}
\end{figure}
\par
In Figure \ref{fig:densh} we show some plots of density $h(x,t)$. 
\par
From Eq.\ (\ref{eq:esprhxt}) we immediately obtain that a gamma-type stationary
density exists when $\tilde \beta<0$.
\begin{corollary}
If $\tilde \beta<0$, then
\begin{equation}
 \overline h(x):=\lim_{t\rightarrow +\infty} h(x,t)
 =\frac{1}{\Gamma(\tilde \gamma)} \Bigg(\frac{| \tilde \beta |}{\tilde\mu}\Bigg)^{\!\tilde \gamma}
 x^{\tilde \gamma-1}\,\exp\Bigg(-x \,\frac{|\tilde \beta|}{\tilde \mu}\Bigg),
 \qquad x\in\mathbb{R}^+.
\label{denstaz}
\end{equation}
\end{corollary}
\par
It is worthwhile to note that the validity of the diffusion approximation discussed in the present 
section is ascertained by comparing the stationary laws of the involved processes. 
Indeed, performing the substitutions (\ref{eq:parametri}) in Eqs.\ (\ref{p0limit}) 
and (\ref{asymptPgrande}) it is not hard to prove that 
$$
 \lim_{\varepsilon \to 0^+}\frac{1}{\varepsilon}\rho(k)\Big|_{k=x/ \varepsilon }
 =\overline h(x),
$$
with $\overline h(x)$ given in (\ref{denstaz}). Similarly, from (\ref{eq:EVarN}) we obtain 
$$
 \lim_{\varepsilon \to 0^+}E[\varepsilon N]=\tilde \gamma\, \frac{\tilde\mu}{| \tilde \beta |}
 =E[X], 
 \qquad 
 \lim_{\varepsilon \to 0^+}Var[\varepsilon N]= \tilde \gamma\,\left( \frac{\tilde\mu}{| \tilde \beta |}\right)^2
 =Var[X],
$$
where $X$ denotes the random variable having density (\ref{denstaz}). 
\section{Discussion}\label{section:new}
In order to discuss some results obtained in the previous sections, we first 
consider $p(0, t)$, i.e.\ the probability of extinction of the population at {\em finite} times $t$. 
This finite-time probability deserves large interest since in many situations researchers cannot 
observe in a reliable manner the population dynamics for very long times. Figures 
\ref{fig:LugualeM}, \ref{fig:LminM} and \ref{fig:LmagM} confirm that $p(0,t)$ decreases 
linearly in $\alpha$ and $d$ when $t$ is close to $0$.  Indeed, from (\ref{eq:system}) 
and (\ref{eq:probiniz}) we have ${d \over d t}\;p(0, t)|_{t=0}=-d \alpha$, and clearly 
${d \over d t}\;P(1, t)|_{t=0}=d \alpha$. Hence, in this multispecies model the number of species 
and the immigration rate play a similar role to increase the survival probability for short times. 
\par
Let us now focus on the stationary distribution obtained in Proposition \ref{stationaryprob} for 
$\lambda <\mu$. From Eq.\ (\ref{asymptPgrande}) we note that $\rho(k)\geq \rho(k+1)$ when 
$k\geq \max\{1,(\alpha-\mu)/(\mu-\lambda)\}$. Hence, if the immigration rate is smaller than the 
death rate then the sequence $\{\rho(k),\; k\in \mathbb{N}^+\}$ is decreasing whatever $d$ is. 
Instead, condition $\rho(0)\geq \rho(1)$ holds when $d \alpha\leq \mu$. This implies 
the following results:  
\par
{\it (a)} \ when $\alpha\leq \mu$ then $\rho(0)\geq \rho(1)\geq \rho(2)\geq\ldots$ if $1\leq d\leq \mu/\alpha$, 
and $\rho(0)\leq \rho(1)\geq \rho(2)\geq\ldots$ otherwise; 
\par
{\it (b)} \ when $\alpha> \mu$ then $\rho(0)\leq \rho(1)\leq \ldots\leq \rho(k^*)\geq \rho(k^*+1)\geq \ldots$ 
for $k^*=\lceil (\alpha-\mu)/(\mu-\lambda)\rceil$. 
\par
In other terms, if $\alpha\leq \mu$ the ratio between the death rate and the immigration rate is a 
critical value for the number of species, since the stationary probability of extinction is larger than 
the stationary probability of any other state $k\in \mathbb{N}^+$, if $1\leq d\leq \mu/\alpha$. 
However, even if $d\geq \mu/\alpha\geq 1$ the stationary distribution attains its maximum for $k=1$. 
Instead, a large immigration rate (i.e., $\alpha> \mu$) yields a larger mode $k^*$ for the stationary 
probability distribution. From Eq.\ (\ref{eq:EVarN}) we have that the mode $k^*$ is very close to 
the stationary expected number of individuals when $d$ is large. This is confirmed, for instance, 
by the contour plots of $k^*=\lceil (\alpha-\mu)/(\mu-\lambda)\rceil$ and $E[N]$ (for $d=10$), 
shown in Figure \ref{fig:kstar} when $0<\lambda/\mu<1<\alpha/\mu<100$. 
%
\begin{figure}[h] 
\begin{center}
\epsfxsize=6.5cm
\epsfbox{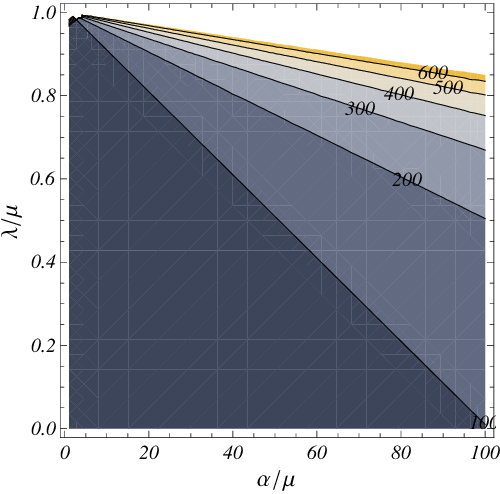}
$\;$
\epsfxsize=6.5cm
\epsfbox{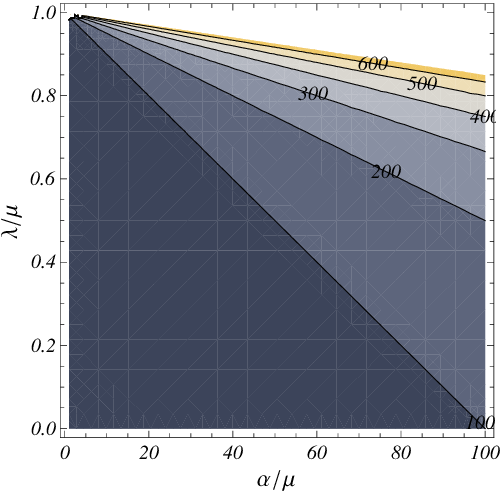}
\caption{Contour plots of $k^*=\lceil (\alpha-\mu)/(\mu-\lambda)\rceil$ (left) and $E[N]$ for $d=10$ (right).}
\label{fig:kstar}
\end{center}
\end{figure}
%
\par
In order to emphasize the dependence on $d$ of the stationary distribution obtained in 
Proposition \ref{stationaryprob}, we note that  $\rho(0)$ is decreasing in $d\in  \mathbb{N}^+$, 
whereas $\rho(k)$ is increasing in $d\in  \mathbb{N}^+$ for $k\in \mathbb{N}^+$. 
Hence, if the number of species increases then the stationary probability of extinction decreases. 
\par
From Eq.\ (\ref{eq:EVarN}) we have that if $d$ grows then $E[N]$ increases, going to a finite 
limit when $d\to +\infty$. In particular, this illustrates that when $\lambda<\mu$ the expected number 
of individuals does not grow indefinitely in the steady state. Moreover, we can adopt the coefficient 
of variation as an adimensional normalized measure of dispersion. For $\lambda<\mu$, we have 
that $CV[N]=\sqrt{Var[N]}/E[N]$ is decreasing in $d \in \mathbb{N}^+$ and tends to a finite limit: 
$$
 \lim_{d\to +\infty}CV[N]=\sqrt{\frac{\mu}{\alpha}-\left(1+\frac{\mu}{\alpha}\right)
 \left(1-\frac{\lambda}{\mu}\right)^{\alpha/\lambda}}.
$$
\par
Another dispersion measure of interest in biological modeling is entropy. 
We recall that the (Shannon) entropy of a random variable $N$ with probability distribution 
$\{\rho(k); k\in\mathbb{N}_0\}$ is defined as 
$$
 H[N]=-\sum_{k=0}^{+\infty}\rho(k)\ln[\rho(k)].
$$
Specifically, $H[N]$ gives the average amount of information that is gained when the 
steady-state number of individuals $N$ is observed. 
In Figure \ref{fig:MedVarCV} we show mean, variance, coefficient of variation and entropy of $N$, 
as a function of $\alpha$, for some choices of the involved parameters. Finally, numerical 
evaluations indicate that $H[N]$ is decreasing in $d\in\mathbb{N}^+$, is decreasing in $\alpha>0$, 
and is decreasing in $\lambda/\mu\in (0,1)$.  
%
\begin{figure}[h] 
\begin{center}
\epsfxsize=6.5cm
\epsfbox{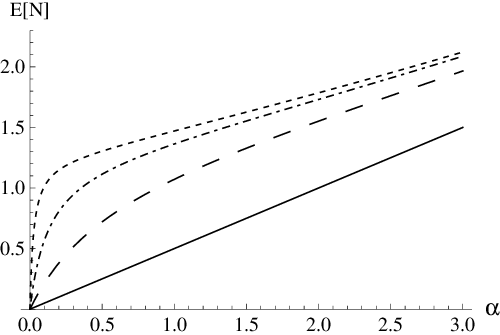}
$\;$
\epsfxsize=6.5cm
\epsfbox{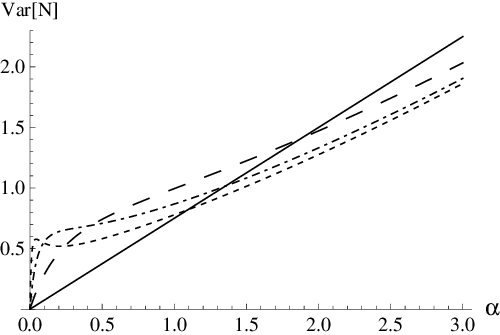}
\\
\epsfxsize=6.5cm
\epsfbox{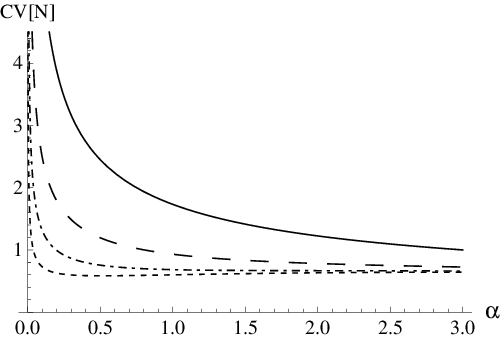}
$\;$
\epsfxsize=6.5cm
\epsfbox{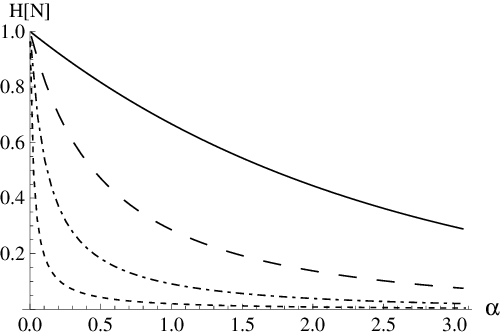}
\caption{Plots of mean, variance, coefficient of variation and entropy of $N$, for $0<\alpha<3$, $\lambda=1$, $\mu=3$ 
and for $d=1$ (plain line), $d=5$ (dashed line), $d=20$ (point-dashed line), $d=100$ (small dashed line).}
\label{fig:MedVarCV}
\end{center}
\end{figure}
\section{Concluding remarks}\label{section:8}
In this paper we focused on the analysis of a continuous-time stochastic process describing 
the dynamics of a population formed by $d$  species competing for a habitat. Formally, we 
considered an extended birth-death-immigration process on a lattice formed by $d$ semiaxes 
joined at the origin. Because of the difficulties in analyzing the stochastic model with different 
rates, we were forced to consider the case of equal transition rates for the various species. 
\par
The main achievements of the paper are concerning: (i) the transient analysis of the process, 
performed by determining the related generating functions and coming to the transient probability 
that the habitat is empty; (ii) the asymptotic distribution of the model, obtained by means of 
Laplace transforms; (iii)  the diffusive approximation of the process, given by a suitable diffusion 
process on the star graph. We point out that the gamma-type stationary density of the approximating 
diffusion process is in tight agreement with the zero-modified negative binomial distribution 
of the original model. 
\par
A thorough discussion on the role of the parameters of the model has also been provided, 
finalized to interpret the given results in biological terms, with special attention to the mode, the 
coefficient of variation and the entropy of the asymptotic distribution. 
\par
We note that the transient analysis of the model have been performed for special 
choices of the parameters, whereas the asymptotic results have been obtained in the 
general case both for the original birth-death model and the approximating diffusion process. 
\appendix
\section{Proof of Theorem \ref{theorem1}}\label{appendix-A}
We now provide the proof of Theorem \ref{theorem1} in $3$ cases.
Recall that the Laplace transform of any function $f(t)$ is denoted as in (\ref{eq:defLsft}).
\subsection{Case $\lambda= \mu$}
From Eq.\ (\ref{eq:eqpot}) if $\lambda= \mu$ we obtain
\begin{equation}
 {\cal L}_s[p(0,t)]={\cal L}_s\left[\frac{1}{(1+\lambda t)^{\frac{\alpha}{\lambda}}}\right]
 -\alpha (d-1) {\cal L}_s[p(0,t)]{\cal L}_s \left[\frac{1}{(1+\lambda t)^{\frac{\alpha}{\lambda}+1}}\right],
 \label{Laplace1LedM}
\end{equation}
where, for any $b\in {\mathbb R}$, $\lambda,s>0$
$$
 {\cal L}_s\left[\frac{1}{(1+\lambda t)^{b}}\right]=\frac{{\rm e}^{s/\lambda}}{\lambda} E\left(b,\frac{s}{\lambda}\right),
$$
and where $E(\nu,z)$ is defined in (\ref{expintfunction}). Noting that
$$
E(\nu,z)=\frac{1}{\nu-1}[{\rm e}^{-z}-z E(\nu-1,z)],\quad \nu\in {\mathbb R},\quad z>0,
$$
from Eq.\ (\ref{Laplace1LedM}) we have
$$
\displaystyle{
{\cal L}_s[p(0,t)]=\frac{1}{\lambda d}\, \frac{{\rm e}^{s/\lambda}\, E\left(\frac{\alpha}{\lambda},\frac{s}{\lambda}\right)}
{1-\frac{d-1}{d}\, \frac{s}{\lambda}\, {\rm e}^{s/\lambda}\, E\left(\frac{\alpha}{\lambda},\frac{s}{\lambda}\right)}.}
$$
Hence, the above expression gives
\begin{eqnarray}
&& \hspace*{-0.8cm}
{\cal L}_s[p(0,t)]=\frac{1}{s d}\sum_{n=0}^{+\infty} \left(1-\frac{1}{d}\right)^{\! n}
\left[\frac{s}{\lambda}\, {\rm e}^{s/\lambda}\, E\left(\frac{\alpha}{\lambda},\frac{s}{\lambda}\right) \right]^{n+1}
\nonumber
\\
&& \hspace*{0.5cm}
=\frac{1}{s d}\sum_{n=0}^{+\infty} \left(1-\frac{1}{d}\right)^{\! n}
\sum_{j=0}^{n+1} {n+1 \choose j} (-1)^j
\left[1-\frac{s}{\lambda}\, {\rm e}^{s/\lambda}\,E\left(\frac{\alpha}{\lambda},\frac{s}{\lambda}\right) \right]^j.
\label{Laplace2LedM}
\end{eqnarray}
Taking the inverse Laplace Transform, from Eq.\ (\ref{Laplace2LedM}) we obtain
\begin{equation}
p(0,t)=1+\frac{1}{d} \sum_{n=0}^{+\infty} \left(1-\frac{1}{d}\right)^n
\sum_{j=1}^{n+1} {n+1 \choose j} (-1)^j F_Y^{(j)}(t),
\label{dimpoLeqM}
\end{equation}
where $F_Y^{(j)}(t)$, defined in (\ref{distrsomma}), is the distribution function of the sum
of $j$ independent random variables having probability density
$$
 f_{Y}^{(1)}(t)=\frac{\alpha}{(1+\lambda t)^{\frac{\alpha}{\lambda}+1}},
 \qquad t>0.
$$
By rearranging the terms in the right-hand side of (\ref{dimpoLeqM}),
Eq.\ (\ref{p0}) immediately follows when $\lambda=\mu$.
\subsection{Case $\lambda> \mu$}
From Eqs.\ (\ref{eq:eqpot}) and (\ref{eq_def1menoGt}) when $\lambda> \mu$ we obtain
\begin{equation}
{\cal L}_s[p(0,t)]\left\{1+ (d-1)
{\cal L}_s
\left[\frac{\alpha (\lambda-\mu)^{\frac{\alpha}{\lambda}+1} {\rm e}^{(\lambda-\mu) t}}{(\lambda {\rm e}^{(\lambda-\mu) t}-\mu)^{\frac{\alpha}{\lambda}+1}}\right]
\right\}
= {\cal L}_s
\left[\frac{(\lambda-\mu)^{\frac{\alpha}{\lambda}}}{(\lambda {\rm e}^{(\lambda-\mu) t}-\mu)^{\frac{\alpha}{\lambda}}}\right].
\label{Laplace1LmagM}
\end{equation}
We have (cf.\ Eq.\ (3.197.3) of Gradshteyn and Ryzhik \cite{GrRy2007})
\begin{eqnarray*}
&& {\cal L}_s\left[\frac{(\lambda-\mu)^{\frac{\alpha}{\lambda}}}{(\lambda {\rm e}^{(\lambda-\mu) t}-\mu)^{\frac{\alpha}{\lambda}}}\right]=\frac{\lambda
\left(\frac{\lambda-\mu}{\lambda}\right)^{\frac{\alpha}{\lambda}}}{\lambda s+(\lambda-\mu)\alpha} 
\,{}_{2}F_{1}\left(\frac{\alpha}{\lambda},\frac{\alpha}{\lambda}+\frac{s}{\lambda-\mu};
1+\frac{\alpha}{\lambda}+\frac{s}{\lambda-\mu};\frac{\mu}{\lambda}\right),
\end{eqnarray*}
where ${}_{2}F_{1}(a,b;c;z)$ is defined in Eq. (\ref{Gaussipergeomfunction}).
Moreover, since
\begin{eqnarray*}
&&{\cal L}_s
\left[\frac{\alpha (\lambda-\mu)^{\frac{\alpha}{\lambda}+1} {\rm e}^{(\lambda-\mu) t}}{(\lambda {\rm e}^{(\lambda-\mu) t}-\mu)^{\frac{\alpha}{\lambda}+1}}\right]=\frac{\lambda \alpha
\left(\frac{\lambda-\mu}{\lambda}\right)^{\frac{\alpha}{\lambda}+1}}{\lambda s+(\lambda-\mu)\alpha} 
\,{}_{2}F_{1}\left(1+\frac{\alpha}{\lambda},\frac{\alpha}{\lambda}+\frac{s}{\lambda-\mu};
1+\frac{\alpha}{\lambda}+\frac{s}{\lambda-\mu};\frac{\mu}{\lambda}\right)
\end{eqnarray*}
and (cf., for instance, Eq.\ 15.2.14 of Abramowitz and Stegun \cite{Abram1994}),
\begin{equation}
b\, {}_{2}F_{1}(a,b+1;c;z)-a\, {}_{2}F_{1}(a+1,b;c;z)+(a-b)\, {}_{2}F_{1}(a,b;c;z)=0,
\label{relGaussIperg}
\end{equation}
from Eq.\ (\ref{Laplace1LmagM}) we have
$$
{\cal L}_s[p(0,t)]\left\{1+ (d-1)
\left[1- \frac{s
\left({\lambda-\mu}\right)^{\frac{\alpha}{\lambda}}\, {}_{2}F_{1}^+}
{\lambda^{\frac{\alpha}{\lambda}-1} (\lambda s+(\lambda-\mu)\alpha)} 
\right]\right\}
= \frac{ \left({\lambda-\mu}\right)^{\frac{\alpha}{\lambda}}\, {}_{2}F_{1}^+}{\lambda^{\frac{\alpha}{\lambda}-1} (\lambda s+(\lambda-\mu)\alpha)},
$$
where, for brevity, we set 
${}_{2}F_{1}^+={}_{2}F_{1}\left(\frac{\alpha}{\lambda},\frac{\alpha}{\lambda}+\frac{s}{\lambda-\mu};
1+\frac{\alpha}{\lambda}+\frac{s}{\lambda-\mu};\frac{\mu}{\lambda}\right)$. 
After some calculations, the above equation gives
\begin{eqnarray}
&& \hspace*{-0.7cm}
{\cal L}_s[p(0,t)]=\frac{1}{s d}\sum_{n=0}^{+\infty} \left(1-\frac{1}{d}\right)^n
\left[\frac{s
\left({\lambda-\mu}\right)^{\frac{\alpha}{\lambda}}
{}_{2}F_{1}^+}
{\lambda^{\frac{\alpha}{\lambda}-1} (\lambda s+(\lambda-\mu)\alpha)}
\right]^{n+1}
\nonumber
\\
&& \hspace*{0.6cm}
=\frac{1}{s d}\sum_{n=0}^{+\infty} \left(1-\frac{1}{d}\right)^n
\sum_{j=0}^{n+1} {n+1 \choose j} (-1)^j
\left[1-\frac{s
\left({\lambda-\mu}\right)^{\frac{\alpha}{\lambda}}
{}_{2}F_{1}^+}{\lambda^{\frac{\alpha}{\lambda}-1} (\lambda s+(\lambda-\mu)\alpha)}\right]^{j}.
\label{Laplace2LmagM}
\end{eqnarray}
Taking the inverse Laplace Transform, in Eq.\ (\ref{Laplace2LmagM}) we get
\begin{equation}
p(0,t)=1+\frac{1}{d} \sum_{n=0}^{+\infty} \left(1-\frac{1}{d}\right)^n
\sum_{j=1}^{n+1} {n+1 \choose j} (-1)^j F_Y^{(j)}(t),
\label{dimpoLmagM}
\end{equation}
where $F_Y^{(j)}(t)$ is the distribution function of the sum of $j$ independent random
variables having probability density
$$
f_{Y}^{(1)}(t)=\frac{\alpha (\lambda-\mu)^{\frac{\alpha}{\lambda}+1} {\rm e}^{(\lambda-\mu) t}}
{\left[\lambda {\rm e}^{(\lambda-\mu) t}-\mu  \right]^{\frac{\alpha}{\lambda}+1}},\quad t>0.
$$
Finally, from Eq. (\ref{dimpoLmagM}) we immediately
obtain Eq.\ (\ref{p0}) when $\lambda>\mu$.
%
\subsection{Case $\lambda< \mu$}
When $\lambda< \mu$, Eq.\ (\ref{Laplace1LmagM}) can be rewritten as
\begin{eqnarray}
&& \hspace*{-0.5cm}
{\cal L}_s[p(0,t)]\Bigg\{1+ (d-1) \left[1-\left(1-\frac{\lambda}{\mu}\right)^{\frac{\alpha}{\lambda}} \right]
\,{\cal L}_s
\Bigg[\frac{\alpha (\lambda-\mu)^{\frac{\alpha}{\lambda}+1} {\rm e}^{(\lambda-\mu) t}}
{\left[1-\left(1-\frac{\lambda}{\mu}\right)^{\frac{\alpha}{\lambda}} \right]
(\lambda {\rm e}^{(\lambda-\mu) t}-\mu)^{\frac{\alpha}{\lambda}+1}}\Bigg]
\Bigg\}
\nonumber
\\
&& \hspace*{4.5cm}
= {\cal L}_s
\left[\frac{(\lambda-\mu)^{\frac{\alpha}{\lambda}}}{(\lambda {\rm e}^{(\lambda-\mu) t}-\mu)^{\frac{\alpha}{\lambda}}}\right],
\label{Laplace1LminM}
\end{eqnarray}
with (cf.\ Eq.\ (3.197.3) of Gradshteyn and Ryzhik \cite{GrRy2007})
$$
{\cal L}_s\left[\frac{(\lambda-\mu)^{\frac{\alpha}{\lambda}}}{(\lambda {\rm e}^{(\lambda-\mu) t}-\mu)^{\frac{\alpha}{\lambda}}}\right]
=\frac{1}{s} \left(1-\frac{\lambda}{\mu}\right)^{\frac{\alpha}{\lambda}}
{}_{2}F_{1}\left(\frac{\alpha}{\lambda},\frac{s}{\mu-\lambda};
1+\frac{s}{\mu-\lambda};\frac{\lambda}{\mu}\right),
$$
and
\begin{eqnarray*}
&& \hspace*{-1.0cm}
{\cal L}_s
\Bigg[\frac{\alpha (\lambda-\mu)^{\frac{\alpha}{\lambda}+1} {\rm e}^{(\lambda-\mu) t}}
{\left[1-\left(1-\frac{\lambda}{\mu}\right)^{\frac{\alpha}{\lambda}} \right]
(\lambda {\rm e}^{(\lambda-\mu) t}-\mu)^{\frac{\alpha}{\lambda}+1}}\Bigg] 
=
\frac{\alpha}{\left[\left(\frac{\mu}{\mu-\lambda}\right)^{\frac{\alpha}{\lambda}}-1 \right]}
\frac{{}_{2}F_{1}\left(1+\frac{\alpha}{\lambda},1+\frac{s}{\mu-\lambda};
2+\frac{s}{\mu-\lambda};\frac{\lambda}{\mu}\right)}{\mu [1+\frac{s}{\mu-\lambda}]}
\nonumber
\\
&& \hspace*{-0.5cm}
=\frac{\lambda \left(\frac{\mu}{\mu-\lambda}\right)^{\frac{\alpha}{\lambda}}}{\mu \left[\left(\frac{\mu}{\mu-\lambda}\right)^{\frac{\alpha}{\lambda}}-1\right]}
\Bigg\{1+\frac{\left[\frac{\alpha}{\lambda}-1-\frac{s}{\mu-\lambda}\right]}{[1+\frac{s}{\mu-\lambda}]\left(\frac{\mu}{\mu-\lambda}\right)^{\frac{\alpha}{\lambda}}}\,  
{}_{2}F_{1}\left(\frac{\alpha}{\lambda},1+\frac{s}{\mu-\lambda};
2+\frac{s}{\mu-\lambda};\frac{\lambda}{\mu}\right)\Bigg\},
\end{eqnarray*}
where use of Eq.\ (\ref{relGaussIperg}) has been made.
Hence, performing some calculations, Eq.\ (\ref{Laplace1LminM}) becomes
\begin{eqnarray*}
&& \hspace*{-0.5cm}
{\cal L}_s[p(0,t)]\left\{1- \frac{\lambda (\mu-\lambda)^{\frac{\alpha}{\lambda}}(d-1)}
{\mu^{\frac{\alpha}{\lambda}}(\mu-\lambda+\lambda d)}
\left[1-\frac{\alpha (\mu-\lambda)}{\lambda (\mu-\lambda+s)}\right]\right. 
= \left.{}_{2}F_{1}\left(\frac{\alpha}{\lambda},1+\frac{s}{\mu-\lambda};
2+\frac{s}{\mu-\lambda};\frac{\lambda}{\mu}\right)
\right\}
\nonumber
\\
&& \hspace*{1cm}
= \frac{\left(1-\frac{\lambda}{\mu}\right)^{\frac{\alpha}{\lambda}}}{s\left[1+\frac{\lambda}{\mu}(d-1)\right]}
\,{}_{2}F_{1}\left(\frac{\alpha}{\lambda},\frac{s}{\mu-\lambda};
1+\frac{s}{\mu-\lambda};\frac{\lambda}{\mu}\right)
\end{eqnarray*}
so that
\begin{eqnarray}
&& \hspace*{-0.1cm}
{\cal L}_s[p(0,t)]=\frac{\mu}{s[\mu+\lambda (d-1)]}
\left(1-\frac{\lambda}{\mu}\right)^{\frac{\alpha}{\lambda}}
\,{}_{2}F_{1}\left(\frac{\alpha}{\lambda},\frac{s}{\mu-\lambda};
1+\frac{s}{\mu-\lambda};\frac{\lambda}{\mu}\right)
\nonumber
\\
&& \hspace*{1.4cm}
\times \sum_{n=0}^{+\infty} \left[\frac{\lambda (d-1)}{\mu+\lambda (d-1)}\right]^n
  \Bigg[\left(1-\frac{\lambda}{\mu}\right)^{\frac{\alpha}{\lambda}}
\left[1-\frac{\alpha (\mu-\lambda)}{\lambda (\mu-\lambda+s)}\right]
\nonumber
\\
&& \hspace*{1.4cm}
\times {}_{2}F_{1}\left(\frac{\alpha}{\lambda},1+\frac{s}{\mu-\lambda};
2+\frac{s}{\mu-\lambda};\frac{\lambda}{\mu}\right)\Bigg]^n.
\label{Laplace3LminM}
\end{eqnarray}
Recalling Eq.\ $15.2.25$ of \cite{Abram1994}, from Eq.\ (\ref{Laplace3LminM}), after some calculations, 
we have
\begin{eqnarray*}
\nonumber
&& \hspace*{-0.5cm}
{\cal L}_s[p(0,t)]=\frac{\lambda}{s[\mu+\lambda (d-1)]}
\sum_{n=0}^{+\infty} \left[\frac{\lambda (d-1)}{\mu+\lambda (d-1)}\right]^n
\left\{
\left(1-\frac{\lambda}{\mu}\right)^{\frac{\alpha}{\lambda}}
\left[1-\frac{\alpha (\mu-\lambda)}{\lambda (\mu-\lambda+s)}\right]
{}_{2}F_{1}^{\star}
\right\}^{n+1}
\nonumber
\\
&& \hspace*{-0.5cm}
+\frac{\mu-\lambda}{s[\mu+\lambda (d-1)]}
\sum_{n=0}^{+\infty} \left[\frac{\lambda (d-1)}{\mu+\lambda (d-1)}\right]^n
\left\{
\left(1-\frac{\lambda}{\mu}\right)^{\frac{\alpha}{\lambda}}
\left[1-\frac{\alpha (\mu-\lambda)}{\lambda (\mu-\lambda+s)}\right]
{}_{2}F_{1}^{\star}
\right\}^{n}
\nonumber
\\
&& \hspace*{-0.5cm}
=\frac{\lambda}{s[\mu+\lambda (d-1)]}
\sum_{n=0}^{+\infty} \left[\frac{\lambda (d-1)}{\mu+\lambda (d-1)}\right]^n
\sum_{j=0}^{n+1}{n+1 \choose j}(-1)^j
\left\{1-
\left(1-\frac{\lambda}{\mu}\right)^{\frac{\alpha}{\lambda}}
\left[1-\frac{\alpha (\mu-\lambda)}{\lambda (\mu-\lambda+s)}\right]
{}_{2}F_{1}^{\star}
\right\}^{j}
\nonumber
\\
&& \hspace*{-0.5cm}
+\frac{\mu-\lambda}{s[\mu+\lambda (d-1)]}
\sum_{n=0}^{+\infty} \left[\frac{\lambda (d-1)}{\mu+\lambda (d-1)}\right]^n
\sum_{j=0}^{n}{n \choose j}(-1)^j
\left\{1-
\left(1-\frac{\lambda}{\mu}\right)^{\frac{\alpha}{\lambda}}
\left[1-\frac{\alpha (\mu-\lambda)}{\lambda (\mu-\lambda+s)}\right]
{}_{2}F_{1}^{\star}
\right\}^{j},
\end{eqnarray*}
where ${}_{2}F_{1}^{\star}={}_{2}F_{1}\left(\frac{\alpha}{\lambda},1+\frac{s}{\mu-\lambda};
2+\frac{s}{\mu-\lambda};\frac{\lambda}{\mu}\right)$. Hence, taking the inverse 
Laplace Transform we get
\begin{eqnarray}
\nonumber
&& \hspace*{-1.cm}
p(0,t)=1+\frac{\lambda}{\mu+\lambda (d-1)}
\sum_{n=0}^{+\infty} \left[\frac{\lambda (d-1)}{\mu+\lambda (d-1)}\right]^n
\sum_{j=1}^{n+1}{n+1 \choose j}\left(-\frac{\mu}{\lambda}\right)^j
\left[1-\left(1-\frac{\lambda}{\mu}\right)^{\frac{\alpha}{\lambda}}\right]^j
F_Y^{(j)}(t)
\nonumber
\\
&& \hspace*{-0.cm}
+\frac{\mu-\lambda}{\mu+\lambda (d-1)}
\sum_{n=0}^{+\infty} \left[\frac{\lambda (d-1)}{\mu+\lambda (d-1)}\right]^n
\sum_{j=1}^{n}{n \choose j}\left(-\frac{\mu}{\lambda}\right)^j
\left[1-\left(1-\frac{\lambda}{\mu}\right)^{\frac{\alpha}{\lambda}}\right]^j
F_Y^{(j)}(t),
\label{InvLaplp0LminM}
\end{eqnarray}
where $F_Y^{(j)}(t)$ is the distribution function of the sum of $j$ independent random
variables having probability density
$$
f_{Y}^{(1)}(t)=\frac{\alpha (\mu-\lambda)^{\frac{\alpha}{\lambda}+1} {\rm e}^{-(\mu-\lambda) t}}
{\left[1-\left(\frac{\mu-\lambda}{\mu}\right)^{\frac{\alpha}{\lambda}}\right]
\left[\mu-\lambda {\rm e}^{-(\mu-\lambda) t}\right]^{\frac{\alpha}{\lambda}+1}},
\qquad t>0.
$$
In conclusion, from Eq.\ (\ref{InvLaplp0LminM}) we obtain Eq.\ (\ref{p0}) when $\lambda<\mu$.
\section*{Acknowledgements}
This research is supported by the biennial research project
``Analytical and stochastical methods for partial differential equations on networks'',
within the  2012-13 Vigoni Program, and  by GNCS-INdAM.
%

%
\end{document}